\documentclass[12pt]{amsart}
\usepackage{amsmath,amsfonts,amssymb,amsthm}
\usepackage{amsmath,amsthm,indentfirst}
\usepackage{amssymb}
\usepackage{amsfonts}
\usepackage{color}
\usepackage{delarray}
\usepackage{amscd}
\usepackage[latin1]{inputenc}
\DeclareMathAlphabet{\mathpzc}{OT1}{pzc}{m}{it}
\newcommand{\mc}[1]{{}}

\theoremstyle{plain}
\newtheorem*{maintheorem*}{Main Theorem}
\newtheorem*{thm*}{Theorem}
\newtheorem*{thma*}{Theorem A}
\newtheorem*{thmaa*}{Theorem A'}
\newtheorem*{thmb*}{Theorem B}
\newtheorem*{thmo*}{Theorem 1.1}
\newtheorem*{thmc*}{Theorem C}
\newtheorem*{thmd*}{Theorem D}
\newtheorem*{thmf*}{Theorem 4.1}
\newtheorem*{remark*}{Remark}
\newtheorem*{conjecture*}{Conjecture}
\newtheorem*{prop*}{Proposition}
\newtheorem*{lem*}{Basic Lemma}

\def\bbr{\mathbb{R}}

\def\bbn{\mathbb{N}}

\def\acal{\mathcal{A}}

\def\hcal{\mathcal{H}}

\def\Bcal{\mathcal{B}}
\def\ical{\mathcal{I}}

\def\Rfrak{\mathfrak{R}}
\def\Ifrak{\mathfrak{I}}

\newcommand{\be}{\begin{equation}}
\newcommand{\ee}{\end{equation}}
\newcommand{\Ext}{\operatorname{Ext}}

\def\h{\hspace{1mm}}
\def\hh{\hspace{.5mm}}

\def\sl{{SL}(2,\bbr)}

\def\rhr{H^1(M,\Sigma,\bbr)}
\def\hr{H^1(M,\bbr)}

\def\strat1{\hcal_1(\alpha)}

\def\supp{{\rm{supp}}}

\newcommand{\bb}{\mathbb}
\newcommand{\gothic}{\mathfrak}

\newcommand{\cx}{{\bb C}}
\newcommand{\half}{{\bb H}}
\newcommand{\integers}{{\bb Z}}
\newcommand{\natls}{{\bb N}}

\newcommand{\reals}{{\bb R}}

\newlength{\figboxwidth}             
\renewcommand{\bold}[1]{\medskip \noindent {\bf #1 }\nopagebreak}

\newcommand{\dirsum}{\oplus}

\newcommand{\tensor}{\otimes}

\newcommand{\cross}{\times}

\newcommand{\st}{\;\: : \;\:}         

\newcommand{\zed}{\integers}

\renewcommand{\mod}{\operatorname{mod}}

\def\@ifundefined#1#2#3%
  {\expandafter\ifx\csname#1\endcsname\relax#2\else#3\fi}

\@ifundefined{theoremstyle}{
}{
\theoremstyle{plain} 
}
\newtheorem{theorem}{Theorem}[section]
\newtheorem{prop}[theorem]{Proposition}
\newtheorem{proposition}[theorem]{Proposition}
\newtheorem{lemma}[theorem]{Lemma}
\newtheorem{cor}[theorem]{Corollary}
\newtheorem{corollary}[theorem]{Corollary}
\newtheorem{claim}[theorem]{Claim}

\@ifundefined{theoremstyle}{
}{
\theoremstyle{definition} 
}
\newtheorem{definition}[theorem]{Definition}

\newtheorem{remark}[theorem]{Remark}

\newcommand{\cA}{{\mathcal A}}
\newcommand{\cB}{{\mathcal B}}
\newcommand{\cC}{{\mathcal C}}

\newcommand{\cH}{{\mathcal H}}
\newcommand{\cI}{{\mathcal I}}

\newcommand{\cL}{{\mathcal L}}
\newcommand{\cM}{{\mathcal M}}
\newcommand{\cN}{{\mathcal N}}

\newcommand{\cS}{{\mathcal S}}

\mathchardef\GG="321D
\newcommand{\gp}{{\gothic p}}

\title[Isolation, equidistribution and orbit closures]{Isolation,
  equidistribution, and orbit closures for the 
  $SL(2,\reals)$ action on  Moduli space. }
\author{Alex Eskin}
\thanks{Research  of  the first author is partially supported  by
NSF grants DMS 0604251, DMS 0905912 and DMS 1201422}
\address{
Department of Mathematics,
University of Chicago,
Chicago, Illinois 60637, USA\\
}
\email{eskin@math.uchicago.edu}

\author{Maryam Mirzakhani}
\thanks{Research of the second author is 
partially supported by the Clay foundation and by NSF grant DMS 0804136}
\address{Department of Mathematics, 
Stanford University, Stanford CA 94305 USA \\}
\email{mmirzakh@math.stanford.edu}

\author{Amir Mohammadi}
\thanks{Research of the third author is 
partially supported by NSF grant DMS 1200388 and Alfred P.\ Sloan Research Fellowship.}
\address{Department of Mathematics, 
University of Texas, 
Austin TX 78712, USA}
\email{amir@math.utexas.edu}

\date{}

\newcommand{\noz}{n}

\begin{document}
\maketitle
\begin{abstract}
We prove results about orbit closures and equidistribution for the
$SL(2,\reals)$ action on the moduli space of compact Riemann surfaces,
which are analogous to the theory of unipotent flows. 
The proofs of the main theorems rely on the measure
classification theorem of \cite{EMir2} and a
certain isolation property of closed $SL(2,\reals)$ invariant
manifolds developed in this paper. 
\end{abstract}

\section{Introduction}
\label{sec;intro}
Suppose $g \ge 1$, and 
let $\alpha = (\alpha_1,\dots, \alpha_\noz)$ be a partition of $2g-2$,
and 
let $\cH(\alpha)$ be a stratum of Abelian differentials,
i.e. the space of pairs $(M,\omega)$ where $M$ is a Riemann surface
and $\omega$ is a holomorphic $1$-form on $M$ whose zeroes have
multiplicities $\alpha_1 \dots \alpha_\noz$. The form $\omega$ defines a
canonical flat metric on $M$ with conical singularities at the zeros
of $\omega$. Thus we refer to points of $\cH(\alpha)$ as
{\em flat surfaces} or {\em translation surfaces}. For an introduction
to this subject, see the survey \cite{Zorich:survey}. 

The space $\cH(\alpha)$
admits an action of the group $SL(2,\reals)$ which generalizes the
action of $SL(2,\reals)$ on the space $GL(2,\reals)/SL(2,\zed)$ of flat
tori.

\bold{Affine measures and manifolds.}
The area of a translation surface is given by 
\begin{displaymath}
a(M,\omega) = \frac{i}{2} \int_M \omega \wedge \bar{\omega}.
\end{displaymath}
A ``unit hyperboloid'' $\cH_1(\alpha)$
is defined as a subset of translation surfaces in $\cH(\alpha)$ of
area one. 
For a subset $\cN_1 \subset \cH_1(\alpha)$ we write
\begin{displaymath}
\reals \cN_1 = \{ (M, t \omega) \;|\; (M,\omega) \in \cN_1, \quad t \in
\reals \} \subset \cH(\alpha).
\end{displaymath}

\begin{definition}
\label{def:affine:measure}
An ergodic $SL(2,\reals)$-invariant probability measure $\nu_1$ on
$\cH_1(\alpha)$ is called {\em affine} if the following hold:
\begin{itemize}
\item[{\rm (i)}] The support $\cM_1$ of $\nu_1$ is an 
{\em immersed submanifold} of
  $\cH_1(\alpha)$, i.e.
there exists a manifold $\cN$ and a proper continuous
  map $f: \cN \to \cH_1(\alpha)$ so that $\cM_1 =
  f(\cN)$. The self-intersection set of $\cM_1$, i.e. the set of
  points of $\cM_1$ which do not have a 
  unique preimage under $f$, is a closed subset of $\cM_1$ of
  $\nu_1$-measure $0$.   Furthermore, each point in $\cN$ has a
  neigborhood $U$ such 
  that locally $\reals f(U)$ is given by a complex linear subspace defined
  over $\reals$ in the period coordinates.
\item[{\rm (ii)}] Let $\nu$ be the measure supported on $\cM = \reals
  \cM_1$ so that $d\nu = 
    d\nu_1 da$. Then each point in $\cN$ has a
      neighborhood $U$ such that the restriction of $\nu$ to $\reals
      f(U)$ is an affine  linear measure in the period
    coordinates on $\reals f(U)$, i.e. it is (up to normalization) the
    restriction of the Lebesgue measure $\lambda$ to the subspace
    $\reals f(U)$.
\end{itemize}
\end{definition}

\begin{definition}
\label{def:affine:invariant:submanifold}
We say that any suborbifold $\cM_1$ for which there exists a measure
$\nu_1$ such that the pair $(\cM_1, \nu_1)$ 
satisfies (i) and (ii) an {\em affine invariant submanifold}. 
\end{definition}
Note that in particular, any affine invariant submanifold is a closed
subset of $\cH_1(\alpha)$ which is invariant under the $SL(2,\reals)$
action, and which in period coordinates looks like an affine subspace.
We also consider the entire stratum {$\cH_1(\alpha)$} 
to be an (improper) affine invariant submanifold.  It follows from
Theorem~\ref{theorem:closed:P:invariant:set} below that the
self-intesection set of an affine invariant manifold is itself a
finite union of affine invariant manifolds of lower dimension.

\bold{Notational Conventions.}
In case there is no confusion, we will often drop the subscript $1$,
and denote an affine manifold by $\cN$. Also we will always denote the
affine probability measure supported on $\cN$ by
$\nu_\cN$. {}{(This measure is unique since it is
  ergodic for the $SL(2,\reals)$ action on $\cN$.}

Let $P \subset SL(2,\reals)$ denote the subgroup $\begin{pmatrix} \ast
  & \ast \\ 0 & \ast \end{pmatrix}$. In this paper we prove statements
about the action of $P$ and $SL(2,\reals)$ on $\strat1$ which are
analogous to the statements proved in the theory of unipotent flows on
homogeneous spaces. \mc{give references}
For some additional results in this direction, see
also \cite{Chaika:Eskin}.

The following theorem is the main result of \cite{EMir2}:
\begin{theorem}
\label{theorem:P:measures}
Let $\nu$ be any $P$-invariant probability measure on
$\cH_1(\alpha)$. Then $\nu$ is $SL(2,\reals)$-invariant and affine.  
\end{theorem}

Theorem~\ref{theorem:P:measures} is a partial analogue of Ratner's
celebrated measure classification theorem in the theory of unipotent flows,
see \cite{RatnerMeas}.


\section{The Main Theorems}
\label{sec:main:theorems}

\subsection{Orbit Closures}
\label{sec:subsec:orbit:closures}

\begin{theorem}
\label{theorem:closure:submanifold}
Suppose $x \in \cH_1(\alpha)$. Then, the orbit closure
$\overline{P x} = \overline{SL(2,\reals) x}$ is an affine invariant
submanifold of $\cH_1(\alpha)$. 
\end{theorem}

The analogue of Theorem~\ref{theorem:closure:submanifold} in the
theory of unipotent flows is due in full generality to M.~Ratner
\cite{RatnerEqui}.  See also the discusion in
\S\ref{sec:subsec:histocial} below.

\begin{theorem}
\label{theorem:closed:P:invariant:set}
Any closed $P$-invariant subset of $\cH_1(\alpha)$ is a finite union
of affine invariant submanifolds. 
\end{theorem}

\subsection{The space of ergodic $P$-invariant measures}

\begin{theorem}
\label{theorem:mozes-shah}
Let $\cN_n$ be a sequence of affine manifolds, and suppose
$\nu_{\cN_n} \to \nu$.  Then $\nu$ is a probability measure.
Furthermore, $\nu$ is the affine measure $\nu_\cN$, where $\cN$ is the
smallest submanifold with the following property: there exists some
$n_0 \in \natls$ such that $\cN_{n}\subset\cN$ for all $n>n_0$.

In particular, the space of ergodic $P$-invariant probability measures on
  $\cH_1(\alpha)$ is compact in the weak-$\ast$ topology. 
\end{theorem}

\begin{remark}
\label{remark:mozes:shah}
In the setting of unipotent flows, {}{the analogue of} 
Theorem~\ref{theorem:mozes-shah} is
due to Mozes and Shah \cite{MS}.
\end{remark}


We state a direct corollary of Theorem~\ref{theorem:mozes-shah}:
\begin{cor}
\label{cor:mozes-shah}
Let $\cM$ be an affine invariant submanifold, and let $\cN_n$ be a
sequence of affine invariant submanifolds of $\cM$  
such that no infinite subsequence is contained in any 
proper affine invariant submanifold of $\cM$. Then the sequence of
affine measures $\nu_{\cN_n}$ converges to $\nu_\cM$.  
\end{cor} 

\subsection{Equidistribution for sectors.}
Let $a_t = \begin{pmatrix} e^t & 0 \\ 0 & e^{-t}  \end{pmatrix}$,
$r_\theta = \begin{pmatrix} \cos \theta & -\sin \theta \\ \sin \theta
  & \cos \theta \end{pmatrix}$. 

\begin{theorem}
\label{theorem:sector:closure}
Suppose $x\in\strat1$ and let $\cM$ be {}{an} 
affine invariant submanifold
of minimum dimension which contains $x$. Then for any $\varphi \in
C_c(\strat1)$, and any interval $I \subset [0,2\pi)$, 
\begin{displaymath}
\lim_{T \to \infty} \frac{1}{T} \int_0^T \frac{1}{|I|} \int_I \varphi
(a_t r_\theta x) \, d\theta \, dt = \int_{\cM} \varphi \, d\nu_{\cM}.   
\end{displaymath}
\end{theorem}

\bold{Remark.} {}{
It follows from Theorem~\ref{theorem:sector:closure}
that for any $x \in \strat1$ there exists a unique affine invariant
manifold of minimal dimension which contains $x$. 
}

We also have the following uniform version:
(cf. \cite[Theorem~3]{Dani:Margulis:distribution})
\begin{theorem}
\label{theorem:sector:uniformity}
Let $\cM$ be an affine invariant submanifold. 
Then for any $\varphi\in C_c(\strat1)$ and any $\epsilon>0$
there are affine invariant submanifolds $\cN_1,\ldots,\cN_\ell$
properly contained in $\cM$ such that 
for any compact subset $F\subset \cM \setminus(\cup_{j=1}^\ell \cN_j)$ 
there exists $T_0$ so that for all $T>T_0$ and any $x\in F$,
\begin{displaymath}
\left|\frac{1}{T} \int_0^T \frac{1}{|I|} \int_I \varphi
(a_t r_\theta x) \, d\theta \, dt -\int_{\cM} \varphi \,
d\nu_{\cM} \right| < \epsilon. 
\end{displaymath}
\end{theorem}

We remark that the analogue of
Theorem~\ref{theorem:sector:uniformity} for unipotent flows, 
due to Dani and Margulis \cite{Dani:Margulis:distribution}
plays a key role in the applications of the theory.

\subsection{Equidistribution for Random Walks}
Let $\mu$ be a probability measure on $SL(2,\reals)$ which is
compactly supported and is absolutely continuous with respect to the
Haar measure. Even though it is not necessary, for clarity of
presentation, we will also assume that 
$\mu$ is $SO(2)$-bi-invariant. Let $\mu^{(k)}$ denote the $k$-fold
convolution of $\mu$ with itself. 

We now state ``random walk'' analogues of
Theorem~\ref{theorem:sector:closure} and
Theorem~\ref{theorem:sector:uniformity}. 

\begin{theorem}
\label{theorem:rw-closure}
Suppose $x\in\strat1,$ and let $\cM$ be the affine invariant submanifold
of minimum dimension which contains $x$. Then for any $\varphi \in
C_c(\strat1)$, 
\begin{displaymath}
\lim_{n \to \infty} \frac{1}{n} \sum_{k=1}^n \int_{SL(2,\reals)} \varphi
(g x) \,
d\mu^{(k)}(g) = \int_{\cM} \varphi \, d\nu_{\cM}.   
\end{displaymath}
\end{theorem}

We also have the following uniform version, similar in spirit to
\cite[Theorem~3]{Dani:Margulis:distribution}: 
\begin{theorem}
\label{theorem:rw-uniformity}
Let $\cM$ be an affine invariant submanifold. 
Then for any $\varphi\in C_c(\strat1)$ and any $\epsilon>0$
there are affine invariant submanifolds $\cN_1,\ldots,\cN_\ell$
properly contained in $\cM$ such that 
for any compact subset $F\subset \cM \setminus(\cup_{j=1}^\ell \cN_j)$ 
there exists $n_0$ so that for all $n>n_0$ and any $x\in F$,
\begin{displaymath}
\left| \frac{1}{n} \sum_{k=1}^n \int_{SL(2,\reals)} \varphi
(g x) \,
d\mu^{(k)}(g) - \int_{\cM} \varphi \, d\nu_{\cM} \right| < \epsilon.
\end{displaymath}
\end{theorem}

\subsection{Equidistribution for some F{\o}lner sets}
Let $u_s = \begin{pmatrix} 1 & s \\ 0 & 1  \end{pmatrix}$.

\begin{theorem}
\label{theorem:folner:closure}
Suppose $x\in\strat1$ and let $\cM$ be the affine invariant submanifold
of minimum dimension which contains $x$. Then for any $\varphi \in
C_c(\strat1)$, and any $r > 0$, 
\begin{displaymath}
\lim_{T \to \infty} \frac{1}{T} \int_0^T \frac{1}{r} \int_0^r \varphi
(a_t u_s x) \, ds \, dt = \int_{\cM} \varphi \, d\nu_{\cM}.   
\end{displaymath}
\end{theorem}

We also have the following uniform version (cf. \cite[Theorem~3]{Dani:Margulis:distribution}):
\begin{theorem}
\label{theorem:folner:uniformity}
Let $\cM$ be an affine invariant submanifold. 
Then for any $\varphi\in C_c(\strat1)$ and any $\epsilon>0$
there are affine invariant submanifolds $\cN_1,\ldots,\cN_\ell$
properly contained in $\cM$ such that 
for any compact subset $F\subset \cM \setminus(\cup_{j=1}^\ell \cN_j)$ 
there exists $T_0$ so that for all $T>T_0$, {}{for all $r >
  0$} and for any $x\in F$,
\begin{displaymath}
\left| \frac{1}{T} \int_0^T \frac{1}{r} \int_0^r \varphi
(a_t u_s x) \, ds \, dt - \int_{\cM} \varphi \, d\nu_{\cM} \right|
< \epsilon.
\end{displaymath}
\end{theorem}

\subsection{Counting periodic trajectories in rational billiards.}
\label{sec:subsec:billiards}
Let $Q$ be a rational polygon, and let $N(Q,T)$ denote the number of
cylinders of periodic trajectories of length at most $T$ 
for the billiard flow on $Q$. By a theorem of H. Masur
\cite{Masur:upper} \cite{Masur:lower}, there exist
$c_1>0$ and $c_2>0$ depending on $Q$ such that for all  $t >1$, 
\begin{displaymath}
c_1 e^{2t} \le N(Q,e^t) \le c_2 e^{2t}. 
\end{displaymath}
As a consequence of Theorem~\ref{theorem:sector:uniformity} we get 
the following ``weak asymptotic
formula'' (cf. \cite{Athreya:Eskin:Zorich}):
\begin{theorem}
\label{theorem:weak:asymptotics}
For any rational polygon $Q$, the exists a constant $c = c(Q)>0$ such
that
\begin{displaymath}
\lim_{t \to \infty} \frac{1}{t} \int_0^{t} N(Q, e^{s}) e^{-2s} \, ds =
c. 
\end{displaymath}
\end{theorem}
The constant $c$ in Theorem~\ref{theorem:weak:asymptotics} 
is the Siegel-Veech constant (see \cite{Veech:Siegel},
\cite{Eskin:Masur:Zorich}) associated to the affine invariant
submanifold $\cM = \overline{SL(2,\reals)S}$ where $S$ is the flat
surface obtained by unfolding $Q$. 

It is natural to conjecture that the extra averaging on
Theorem~\ref{theorem:weak:asymptotics} is not necessary, and one has
$\lim_{t \to \infty} N(Q,e^t) e^{-2t}
= c$. This can be shown if one obtains a classification of the
measures invariant under the subgroup $U = \begin{pmatrix} 1 & \ast \\
  0 & 1 \end{pmatrix}$ of $SL(2,\reals)$. Such a
result is in general beyond the reach of the current methods. However
it is known in a few very special cases, 
see \cite{Eskin:Masur:Schmoll}, \cite{Eskin:Marklof:Morris},
\cite{Calta:Wortman} and \cite{Bainbridge:L:shaped}. 
\medskip

\subsection{The Main Proposition and Countability}
For a function $f: \cH_1(\alpha) \to \reals$, let
\begin{displaymath}
(A_t f)(x) = \frac{1}{2\pi} \int_0^{2\pi} f(a_t r_\theta x). 
\end{displaymath}
Following the general idea of Margulis {}{introduced
  in\cite{EMM1}}, 
the strategy of the proof is
to define a function which will satisfy a certain inequality involving
$A_t$. In fact,
the main technical result of this paper is the following:
\begin{proposition}
\label{prop:main:proposition}
Let $\cM \subset \cH_1(\alpha)$ be an affine invariant submanifold. (In this
proposition $\cM = \emptyset$ is allowed). Then there
exists an $SO(2)$-invariant 
function $f_\cM: \cH_1(\alpha) \to [1,\infty]$ with the following
properties:
\begin{itemize}
\item[{\rm (a)}] $f_\cM(x) = \infty$ if and only if $x \in \cM$, and
  $f_\cM$ is bounded on compact subsets of $\strat1\setminus \cM$. 
   For any $\ell > 0$, the set $\overline{\{ x \st f_\cM(x) \le \ell \}}$ is a
  compact subset of $\cH_1(\alpha)\setminus \cM$. 
\item[{\rm (b)}] There exists  $b > 0$ (depending on $\cM$) and for
  every $0 < c < 1$   there exists $t_0 > 0$ (depending on $\cM$ and
  $c$) such that for all $x \in \cH_1(\alpha)\setminus \cM$ and all $t > t_0$,
\begin{displaymath}
(A_t f_\cM)(x) \le c f_\cM(x) + b.
\end{displaymath}
\item[{\rm (c)}] There exists $\sigma > 1$ such that for all $g \in
  SL(2,\reals)$ {}{in some neighborhood of the identity}
and all $x \in \cH_1(\alpha)$, 
\begin{displaymath}
\sigma^{-1} f_\cM(x) \le f_\cM(g x) \le \sigma f_\cM(x).
\end{displaymath}
\end{itemize}
\end{proposition}
The proof of Proposition~\ref{prop:main:proposition} consists of
\S\ref{sec:recurrence}-\S\ref{sec:function}. It is based on 
the recurrence properties of the $SL(2,\reals)$-action proved by
Athreya in \cite{A}, and also the fundamental result of Forni on the
uniform hyperbolicity in compact sets of the Teichm\"uller geodesic
flow \cite[Corollary~2.1]{Forni}.

\begin{remark}
\label{remark:M:empty}
In the case $\cM$ is empty, a function satisfying the conditions of 
Proposition~\ref{prop:main:proposition} has been constructed in \cite{EMas}
and used in \cite{A}. 
\end{remark}

\begin{remark}
\label{remark:complexity}
In fact, we show that the constant $b$ in
Proposition~\ref{prop:main:proposition} (b) depends only on the
``complexity'' of $\cM$ (defined in \S\ref{sec:regions}). This is used in
\S\ref{sec:countability} for the proof of the following:
\end{remark}

\begin{proposition}
\label{prop:countability}
There are at most countably many affine invariant submanifolds in each stratum.
\end{proposition}

Another proof of Proposition~\ref{prop:countability} is given in
\cite{Wright:numberfield}, where it is shown that any affine invariant
submanifold
is defined over a number field.

\subsection{Analogy with unipotent flows and historical remarks}
\label{sec:subsec:histocial}
In the context of unipotent flows, i.e. the left-multiplication 
action of a unipotent
subgroup $U$ of a Lie group $G$ on the space $G/\Gamma$ where $\Gamma$
is a lattice in $G$, the analogue of
Theorem~\ref{theorem:closure:submanifold} was conjectured by
Raghunathan. 
In the literature the conjecture was first
stated in the paper \cite{Dani:horospherical:1981} and in a more general
form in \cite{Margulis:number} (when the subgroup $U$ is not
necessarily unipotent but generated by unipotent elements). 
Raghunathan's conjecture was eventually proved in full generality
by M.~Ratner (see \cite{RatnerSolv},
\cite{RatnerSS},
\cite{RatnerMeas} and \cite{RatnerEqui}).  
Earlier it was known in the following cases:
(a) $G$ is reductive and $U$ is horospherical (see
\cite{Dani:horospherical:1981}); (b) $G = SL(3,\reals)$ and $U = \{ u(t)\}$
is a one-parameter unipotent subgroup of $G$ such that $u(t) - I$ has
rank 2 for all $t \ne 0$, where $I$ is the identity matrix (see
\cite{Dani:Margulis:unipotent}); (c) $G$ is solvable (see
\cite{Starkov:one} and \cite{Starkov:two}). 
We remark that 
the proof given in \cite{Dani:horospherical:1981} is restricted to
horospherical $U$ and  the proof given in \cite{Starkov:one} and
\cite{Starkov:two} cannot be applied for nonsolvable $G$.

However the proof in \cite{Dani:Margulis:unipotent}
together with the methods developed in \cite{Mar2}, \cite{Mar3},
\cite{Mar4} and \cite{Dani:Margulis:values}
suggest an approach for
proving the Raghunathan conjecture in general by studying the minimal
invariant sets, and the limits of orbits of sequences of points tending
to a minimal invariant set.  This program was being
actively pursued at the time Ratner's results were announced
(cf. \cite{Shah:thesis}).

\section{Proofs of the Main Theorems}

In this section we derive all the results of
\S\ref{sec:subsec:orbit:closures}-\S\ref{sec:subsec:billiards} 
from Theorem~\ref{theorem:P:measures},
Proposition~\ref{prop:main:proposition} and
Proposition~\ref{prop:countability}. 

The proofs are much simpler then the proofs of the analogous results
in the theory of unipotent flows. This is related to
Proposition~\ref{prop:countability}. In the setting of unipotent flows
there may be continuous families of invariant manifolds (which involve
the centralizer and normalizer of the acting group).

\subsection{Random Walks}
\label{sec:subsec:proof:random:walks}
Many of the arguments work most naturally in the random walk
setting. But first we need to convert Theorem~\ref{theorem:P:measures}
and Proposition~\ref{prop:main:proposition} to the random walk setup. 

\bold{Stationary measures.}
Recall that $\mu$ is a compactly supported probability measure on $SL(2,\reals)$
which is $SO(2)$-bi-invariant and is absolutely continuous with
respect to Haar measure.  A measure $\nu$ on $\cH_1(\alpha)$ is called 
{\em $\mu$-stationary} if $\mu * \nu = \nu$, where
\begin{displaymath}
\mu * \nu = \int_{SL(2,\reals)} (g_* \nu) \, d\mu(g). 
\end{displaymath}

Recall that by a theorem of Furstenberg \cite{F1}, \cite{F2}, restated
as \cite[Theorem 1.4]{Nevo:Zimmer}, $\mu$-stationary measures are in
one-to-one correspondence with $P$-invariant measures. Therefore, 
Theorem~\ref{theorem:P:measures} can be reformulated as
the following:
\begin{theorem}
\label{theorem:stationary:invariant}
Any $\mu$-stationary measure on $\cH_1(\alpha)$ is
$SL(2,\reals)$ invariant and affine. 
\end{theorem}
\bold{The operator $\mathbb{A}_\mu$.}
Let $\mathbb{A}_\mu: C_c(\strat1) \to C_c(\strat1)$ denote the linear
operator 
\begin{displaymath}
(\mathbb{A}_\mu f)(x) = \int_{SL(2,\reals)} f(g x) \, d\mu(g)
\end{displaymath}

\begin{lemma}
\label{lemma:ineq:random:walk}
Let $f_\cM$ be as in Proposition~\ref{prop:main:proposition}. Then 
there exists $b > 0$ and for any $c > 0$ there exists $n_0 > 0$ such
that for $n > n_0$, and any $x \in \strat1$, 
\begin{displaymath}
(\mathbb{A}_\mu^n f_\cM)(x) \le c f_\cM(x) + b.
\end{displaymath}
\end{lemma}

\begin{proof} 
Recall the $KAK$ decomposition:
\begin{displaymath}
g = k_1 a_t k_2, \qquad \text{ $g \in SL(2,\reals)$, $k_1, k_2 \in
  SO(2)$, $t \in \reals^+$}.
\end{displaymath}
We may think of $k_1,t,k_2$ as coordinates on $SL(2,\reals)$. 
Since $\mu^{(n)}$ is $SO(2)$-bi-invariant and absolutely continuous
with respect to the Haar measure on $SL(2,\reals)$, we have
\begin{displaymath}
d\mu^{(n)}(g) = K_n(t) \,dm(k_1) \,dm(k_2) \, dt,
\end{displaymath}
where $m$ is the Haar measure on $SO(2)$, and $K_n: \reals^+ \to
\reals$ is a compactly supported function satisfying $K_n(t) \ge 0$,
$\int_0^\infty K_n(t) \, dt = 1$. Also, 
{}{since the top Lyapunov exponent of the random walk on
  $SL(2,\reals)$ given by $\mu$ is positive,} 
for any $t_0 > 0$ and any
$\epsilon > 0$ there exists
$n_0$ such that for $n > n_0$, 
\begin{equation}
\label{eq:Kn:kernel}
\int_0^{t_0} K_n(t) \, dt < \epsilon. 
\end{equation}
We have, since $f_\cM$ is $SO(2)$-invariant,
\begin{equation}
\label{eq:def:Kn}
(\mathbb{A}_\mu^n f_\cM)(x) = \int_0^\infty K_n(t) (A_t f_\cM)(x) \, dt,
\end{equation}
Now let $t_0$ be as in Proposition~\ref{prop:main:proposition} (b) for
$c/2$ instead of $c$. By Proposition~\ref{prop:main:proposition} (c),
there exists $R > 0$ such that 
\begin{equation}
\label{eq:tmp:R}
f_\cM(a_t r_\theta x) < R f_\cM(x) \quad \text{ when $t <
t_0$.}
\end{equation}
Then let $n_0$ be such that (\ref{eq:Kn:kernel}) holds with
$\epsilon = c/(2R)$. Then, for $n > n_0$, 
\begin{align*}
(\mathbb{A}_\mu^n f_\cM)(x) & = \int_0^{t_0} K_n(t) (A_t f_\cM)(x) \, dt +
\int_{t_0}^\infty K_n(t) (A_t f_\cM)(x) \, dt && \text{by (\ref{eq:def:Kn})} \\
& \le \int_0^{t_0} K_n(t) (R f_\cM(x)) \, dt + \int_{t_0}^\infty ((c/2) f_\cM(x) +
b) \, dt && \text{by (\ref{eq:tmp:R}) and
  Proposition~\ref{prop:main:proposition} (b)}\\
& \le (c/2R) R f_\cM(x) + ((c/2)f_\cM(x) + b)  && \text{by
  (\ref{eq:Kn:kernel})} \\
& = c f_\cM(x) + b. 
\end{align*}
\end{proof}

\bold{Notational conventions.}
Let
\begin{displaymath}
\bar{\mu}^{(n)} = \frac{1}{n} \sum_{k=1}^n \mu^{(k)}. 
\end{displaymath}
For $x \in \cH_1(\alpha)$ let $\delta_x$ denote the Dirac measure at
$x$, and let $\ast$ denote convolution of measures.

We have the following: 
\begin{prop}
\label{prop:rw-avoid-singular-set}
Let $\cN$ be a (possibly empty) proper affine invariant submanifold. 
Then for any $\epsilon>0$, 
there exists an open set $\Omega_{\cN,\epsilon}$ containing $\cN$ with
$(\Omega_{\cN,\epsilon})^c$ compact such that for any
compact $F\subset\strat1\setminus\cN$ there exists $n_0 \in \natls$
so that for all $n>n_0$ and all $x\in F$, we have
\begin{displaymath}
(\bar{\mu}^{(n)} \ast \delta_x)(\Omega_{\cN,\epsilon})<\epsilon.
\end{displaymath}
\end{prop}

\begin{proof}
Let $f_\cN$ 
be the function of Proposition~\ref{prop:main:proposition}.  Let $b >
0$ be as in Lemma~\ref{lemma:ineq:random:walk}, and let
\begin{displaymath}
\Omega_{\cN,\epsilon}=\left\{p: f_\cN( p)>(b+1)/\epsilon \right\}^0,
\end{displaymath} 
where $E^0$ denotes the interior of $E$. 

Suppose $F$ is a compact subset of $\strat1 \setminus \cN$. 
Let $m_F=\sup\{f_\cN(x): x\in F\}$. 
Let $n_0 \in \natls$ be as in Lemma~\ref{lemma:ineq:random:walk} for $c =
{}{0.5}/m_F$. Then, by Lemma~\ref{lemma:ineq:random:walk},
\begin{displaymath}
(\mathbb{A}_{\mu}^{n} f_\cN)(x)<\frac{{}{0.5}}{m_F}
f_\cN(x)+ b\leq {}{0.5} + b, 
\qquad \text{for all $n > n_0$ and all $x\in F.$}
\end{displaymath}
{}{It follows that for $n_0$ sufficiently large,}
for all $x \in F$ and all $n > n_0$, 
\begin{displaymath}
(\bar{\mu}^{(n)} \ast \delta_x)(f_\cN) \le 1+b. 
\end{displaymath}
Thus for any $x\in F$ and $L>0$ we have 
\begin{equation}
\label{eq:eq:avoid}
(\bar{\mu}^{(n)} \ast \delta_x)(\{p: f_\cN( p)>L\})<\frac{b+1}{L}. 
\end{equation}
Then (\ref{eq:eq:avoid}) implies that
$
(\bar{\mu}^{(n)} \ast \delta_x)(\Omega_{\cN,\epsilon})<\epsilon.
$
Also, Proposition~\ref{prop:main:proposition} (a) 
implies that $\Omega_{\cN,\epsilon}$ is a neighborhood of $\cN$ and
\begin{displaymath}
(\Omega_{\cN,\epsilon})^c = \overline{\{ p \st f_\cN(p) \le
(b+1)/\epsilon \}} 
\end{displaymath}
is compact.
\end{proof}

\begin{proof}[Proof of Theorem~\ref{theorem:rw-closure}]
Let $\cM$ be an affine manifold contaning $x$ of
minimal dimension. (At this point we do not yet know that $\cM$ is
unique). Suppose the assertion of the theorem does not hold.
Then there exist a 
$\varphi\in C_c(\strat1)$, $\epsilon>0$,
$x \in \cM$ and a sequence $n_k \to\infty$ such that 
\begin{displaymath}
|(\bar{\mu}^{(n_k)} \ast \delta_{x})(\varphi)-\nu_\cM(\varphi)|\geq\epsilon.
\end{displaymath}
Recall that the space of measures on $\strat1$ of total mass at most $1$ is
compact in the weak star topology. Therefore, 
after passing to a subsequence if necessary, 
we may and will assume that $\bar{\mu}^{(n_k)} \ast \delta_{x} \to
\nu$ where $\nu$ is some measure on $\strat1$ (which could a priori be 
the zero measure). Below, we will show that in fact $\nu$ is
the probability measure $\nu_\cM$, which leads to a contradiction.

First note that it follows from the definition that $\nu$ 
is an $\mu$-stationary measure. Therefore, by
Theorem~\ref{theorem:stationary:invariant}, $\nu$ is
$SL(2,\reals)$-invariant. 
Also since $\cM$ is $\sl$-invariant we get $\supp(\nu)\subset\cM$. The
measure $\nu$ need not be ergodic, but by
Theorem~\ref{theorem:P:measures}, all of its ergodic components are
affine measures supported on affine invariant submanifolds of $\cM$. 
By Proposition~\ref{prop:countability} there are only countably many 
affine invariant submanifolds of $\cM$. Therefore, we have the ergodic
decomposition:
\begin{equation}
\label{eq:ergodic-decomp}
\nu=\sum_{\cN\subseteq\cM}a_\cN \nu_\cN,
\end{equation}
where the sum is over all the 
affine invariant submanifolds $\cN \subset \cM$ and $a_\cN \in [0,1]$. 
To finish the proof we will show that $\nu$ is a probability measure,
and that  $a_\cN=0$ for all $\cN\subsetneq\cM$. 

Suppose $\cN\subsetneq\cM$. {}{(Here we allow $\cN =
\emptyset$)}. 
Note that $x \not\in \cN$ (since $\dim \cN < \dim \cM$ and $\cM$ is
assumed to be an affine manifold containg $x$ of minimal dimension). 
We now apply 
Proposition~\ref{prop:rw-avoid-singular-set}
with $\cN$ and the compact set $F = \{x\}$. We get for any $\epsilon>0,$ 
there exists some $n_0$ so that if $n>n_0$,
then $(\bar{\mu}^{(n)} \ast \delta_x)((\Omega_{\cN,\epsilon})^c) \ge
1-\epsilon$.  Therefore, passing to the limit, we get
\begin{displaymath}
\nu((\Omega_{\cN,\epsilon})^c) \ge 1-\epsilon.
\end{displaymath}
{}{Note that $\epsilon > 0$ is arbitrary. From the case $\cN =
  \emptyset$ we get that $\nu$ is a probability measure. Also for any
  $\cN \subsetneq \cM$ this implies that $\nu(\cN) = 0$. Hence $a_\cN \le
  \nu(\cN) = 0$. 
}
\end{proof}

\begin{proof}[Proof of Theorem~\ref{theorem:mozes-shah}]

Since the space of measures of mass at most $1$ on $\strat1$ is
compact in the weak-$\ast$ topology, 
the {}{statement about weak-$\ast$ compactness in
  Theorem~\ref{theorem:mozes-shah} follows from the others.} 

Suppose that $\nu_{\cN_n} \to \nu$. We first prove that
$\nu$ is a probability measure. Let
$\Omega_{\emptyset,\epsilon}$ be as in
Proposition~\ref{prop:rw-avoid-singular-set} with $\cM = \emptyset$.
By the random ergodic theorem \cite[Theorem~3.1]{Furman:Survey}, 
for a.e $x_n \in \cN_n$,  
\begin{equation}
\label{eq:tmp:random:ergodic}
\lim_{m\to\infty}(\bar{\mu}^{(m)} \ast
\delta_{x_n})((\Omega_{\emptyset,\epsilon})^c) =
\nu_{\cN_n}((\Omega_{\emptyset,\epsilon})^c). 
\end{equation}
Choose $x_n$ such that (\ref{eq:tmp:random:ergodic}) holds. 
By Proposition~\ref{prop:rw-avoid-singular-set}, for all $m$ large
enough (depending on $x_n$), 
\begin{displaymath}
(\bar{\mu}^{(m)} \ast
\delta_{x_n})((\Omega_{\emptyset,\epsilon})^c) \ge 1-\epsilon. 
\end{displaymath}
Passing to the limit as $n \to \infty$, we get
\begin{displaymath}
\nu((\Omega_{\emptyset,\epsilon})^c) \ge 1-\epsilon. 
\end{displaymath}
Since $\epsilon$ is arbitrary, this shows that $\nu$ is a probability
measure. 

In view of the fact that the $\nu_n$ are invariant under
$\sl$, the same is true of $\nu$. As in (\ref{eq:ergodic-decomp}), let
\begin{displaymath}
\nu=\sum_{\cN\subseteq\strat1} a_\cN \nu_{\cN}
\end{displaymath}
be the ergodic decomposition of $\nu$, where
$a_\cN \in [0,1]$. By Theorem~\ref{theorem:P:measures}, all the
measures $\nu_\cN$ are affine and by
Proposition~\ref{prop:countability}, the number of terms in the
ergodic decomposition is countable.

For any affine invariant submanifold $\cN$ let 
\begin{displaymath}
X(\cN)=\bigcup\left\{\cN'\subsetneq\cN:\cN'\h
\mbox{is an affine invariant submanifold}\right\}.
\end{displaymath}
Let $\cN\subseteq\strat1$ be a submanifold  such that
$\nu(X(\cN))=0$ and $\nu(\cN) > 0$.
This implies $a_\cN=\nu(\cN)$. 

Let $K$ be a large compact set, such that $\nu(K) >
(1-a_\cN/4)$. Then, $\nu(K \cap \cN) > (3/4)a_\cN$. 
Let $\epsilon = a_\cN/4$, and let $\Omega_{\cN,\epsilon}$ be as in 
Proposition~\ref{prop:rw-avoid-singular-set}. Since $K \cap \cN$ and
$(\Omega_{\cN,\epsilon})^c$ are both compact sets, we can choose a
continuous compactly supported function $\varphi$ such that $0 \le
\varphi \le 1$, $\varphi = 1$ on $K \cap \cN$ and $\varphi = 0$ on
$(\Omega_{\cN,\epsilon})^c$. Then, 
\begin{displaymath}
\nu(\varphi) \ge \nu(K \cap \cN) > (3/4)a_\cN.
\end{displaymath}
Since $\nu_{\cN_n}(\varphi) \to \nu(\varphi)$, there exists $n_0 \in
\natls$ such that for $n  > n_0$, 
\begin{displaymath}
\nu_{\cN_n}(\varphi) > a_\cN/2. 
\end{displaymath}
For each $n$ let $x_n\in\cN_n$ be a generic point for $\nu_{\cN_n}$ 
for the random ergodic theorem \cite[Theorem~3.1]{Furman:Survey} i.e.
\begin{equation}
\label{eq:random-ergodic-average}
\text{$\lim_{m\to\infty}(\bar{\mu}^{(m)} \ast \delta_{x_n})(\varphi)=
\nu_{\cN_n}(\varphi)$ 
for all $\varphi\in C_c(\strat1).$
}
\end{equation}
Suppose  $n>n_0$. Then, by~\eqref{eq:random-ergodic-average}, 
we get 
\begin{displaymath}
\mbox{if $m$ is large enough, then  
$(\bar{\mu}^{(m)} \ast \delta_{x_n})(\varphi)>a_\cN/4.$}
\end{displaymath} 
Therefore, since $0 \le \varphi \le 1$ and $\varphi = 0$ outside of
$\Omega_{\cN,\epsilon}$, we get
\begin{displaymath}
\mbox{if $m$ is large enough, then  
$(\bar{\mu}^{(m)} \ast \delta_{x_n})(\Omega_{\cN,\epsilon})>a_\cN/4.$}
\end{displaymath} 

Proposition~\ref{prop:rw-avoid-singular-set}, applied with $\epsilon=
a_\cN/4$  now implies that
$x_n\in\cN,$ which, in view of the genericity of $x_n$ implies 
that $\cN_n\subset\cN$ for all $n>n_0.$ This implies $\nu(\cN)=1$,
and since $\nu(X(\cN))=0,$ we get $\nu=\nu_\cN$. Also, since
$\nu(X(\cN))=0$, $\cN$ is the minimal affine invariant manifold which
eventually contains the $\cN_n$. 
\end{proof}

\begin{lemma}
\label{lemma:finite-collection}
Given any $\varphi\in C_c(\strat1)$, any affine invariant submanifold
$\cM$ and any $\epsilon>0$,
there exists a finite collection  $\cC$ of proper affine invariant
submanifolds of $\cM$ with the following property:
if $\cN' \subset \cM$ is such that 
$|\nu_{\cN'}(\varphi)-\nu_\cM(\varphi)| \geq \epsilon,$ then
there exists some $\cN\in\cC$ such that $\cN'\subset\cN.$
\end{lemma}

\proof
Let $\varphi$ and $\epsilon>0$ be given.
We will prove this by inductively choosing $\cN_j$'s as follows.
Suppose $k>0,$ and put 
\begin{displaymath}
\acal_k=\{\cN \subseteq\cM:\cN\text{ has codimension } 
k\text{ in }\cM\text{ and }
|\nu_\cN(\varphi)-\nu_\cM(\varphi)|\geq\epsilon\}.
\end{displaymath}
Let $\Bcal_1=\acal_1,$ and define 
\begin{displaymath}
\Bcal_k=\{\cN\in\acal_k: \text{such that }\cN 
\text{ is not contained in any }\cN'\in\acal_\ell
\text{ with }\ell<k\}.
\end{displaymath}
{\it Claim.} $\Bcal_k$ is a finite set for each $k.$ 

We will show this inductively. Note that by 
Corollary~\ref{cor:mozes-shah} we have $\acal_1,$
and hence $\Bcal_1,$ is a finite set. Suppose we have shown 
$\{\Bcal_j:1\leq j\leq k-1\}$ is a finite set.
Let $\{\cN_j\}$ be an infinite collection
of elements in $\Bcal_k.$ 
By Theorem~\ref{theorem:mozes-shah}
we may pass to subsequence (which we continue to denote by $\cN_j$)  
such that $\nu_{\cN_j}\rightarrow\nu$. 
Theorem~\ref{theorem:mozes-shah} also implies that $\nu=\nu_\cN$ for
some affine invariant submanifold $\cN$, and that
there exists some $j_0$ such that $\cN_j\subset\cN$
for all $j>j_0$. Note that $\cN$ has codimension $\ell\leq k-1$.

Since $\nu_{\cN_j} \to \nu_\cN$, and $\cN_j \in \cB_k
\subset \cA_k$, we have $|\nu_\cN(\varphi) - \nu_\cM(\varphi)| \ge
\epsilon$. Therefore $\cN \in \cA_\ell$. But this is a contradiction
to the definition of $\cB_k$ since $\cN_j \subset \cN$ and $\cN_j \in
\cB_k$. This completes the proof of the claim. 

Now let
\begin{displaymath}
\cC=\{\cN:\cN\in\Bcal_k,\text{ for }0<k\leq \dim\cM\}.
\end{displaymath}
This is a finite set which satisfies the conclusion of the lemma.
\qed

\begin{proof}[Proof of Theorem~\ref{theorem:rw-uniformity}]
Let $\varphi$ and $\epsilon>0$ be given,
and let $\cC$ be given by Lemma~\ref{lemma:finite-collection}.
Write $\cC = \{\cN_1, \dots,\cN_\ell\}$. 
We will show the theorem holds with this choice of the $\cN_j$. 

Suppose not, then there exists a compact subset 
$F\subset {}{\cM}\setminus\bigcup_{j=1}^\ell \cN_j$ 
such that for all $m_0\geq 0$,
\begin{displaymath}
\{x\in F:|(\bar{\mu}^{(m)} \ast
\delta_x)(\varphi)-\nu_\cM(\varphi)|>\epsilon,\text{ for some
}m>m_0\}\neq\emptyset. 
\end{displaymath}
Let $m_n\rightarrow\infty$ and $\{x_n\}\subset F$ 
be a sequence such that 
$|(\bar{\mu}^{(m_n)} \ast
\delta_{x_n})(\varphi)-\nu_\cM(\varphi)|>\epsilon$. 

Since the space of measures on $\strat1$ of total mass at most $1$ is
compact in the weak star topology,  
after passing to a subsequence if necessary, 
we may and will assume that $\bar{\mu}^{(m_n)} \ast \delta_{x_n} \to
\nu$ where $\nu$ is some measure on $\cM$ (which could a priori be 
the zero measure). We will also assume that $x_n\rightarrow x$
for some $x \in F$.

Note that $\nu$ is $\sl$-invariant. Let
\begin{displaymath}
\nu=\sum_{\cN\subseteq\cM}a_\cN \nu_\cN
\end{displaymath}
be the ergodic decomposition of $\nu$, as in
(\ref{eq:ergodic-decomp}). 

We claim that $\nu$ is a probability measure and
$\nu(\cN)=0$ for all
$\cN\in\cC.$ To see this, suppose $\cN\in\cC$ {}{or $\cN
  = \emptyset$} and apply 
Proposition~\ref{prop:rw-avoid-singular-set}
with $\cN$ and $F$. We get for any $\epsilon'>0$ 
there exists some $n_0$ so that if $n>n_0$,
then $(\bar{\mu}^{(n)} \ast \delta_y)((\Omega_{\cN,\epsilon'})^c) \ge
1-\epsilon'$ for all $y \in F$.  Therefore, passing to the limit, we get
\begin{displaymath}
\nu((\Omega_{\cN,\epsilon})^c) \ge 1-\epsilon'.
\end{displaymath}
Since $\epsilon' > 0$ is arbitrary, this implies that $\nu$ is a
probability measure and $\nu(\cN)=0$.
The claim now follows since $\cC$ is a finite family. 

The claim \mc{give number} and Lemma~\ref{lemma:finite-collection}
imply that
$|\nu(\varphi)-\nu_\cM(\varphi)|<\epsilon$. This and the definition
of $\nu$ imply that
$|(\bar{\mu}^{(m_n)} \ast \delta_{x_n})(\varphi)-\nu_\cM(\varphi)|<\epsilon$ 
for all large enough $n$. This contradicts the choice of $x_n$ and $m_n$ 
and completes the proof. 
\end{proof}

The only properties of the measures $\bar{\mu}^{(n)}$ which
were used in this subsection were
Proposition~\ref{prop:rw-avoid-singular-set}  and the fact that any
limit of the measures $\bar{\mu}^{(n)} \ast \delta_x$ is
$SL(2,\reals)$ invariant. In fact, we proved the following theorem, 
which we will record for future use:
\begin{theorem}
\label{theorem:general:measures}
Suppose $\{\eta_t \st t \in \reals\}$ is a family of probability 
measures on $SL(2,\reals)$ with the following properties:
\begin{itemize}
\item[{(a)}] Proposition~\ref{prop:rw-avoid-singular-set} holds for
  $\eta_t$ instead of $\bar{\mu}^{(n)}$ (and $t$ instead of
  $n$). 
\item[{(b)}] Any weak-$\ast$ limit of measures of the form $\eta_{t_i} \ast
  \delta_{x_i}$ as $t_i \to \infty$ is $SL(2,\reals)$-invariant. 
\end{itemize}
Then,
\begin{itemize}
\item[{\rm (i)}] (cf. Theorem~\ref{theorem:rw-closure}) 
Suppose $x \in \strat1$, and let $\cM$ be the smallest affine
invariant submanifold containing $x$. 
Then for any $\varphi\in C_c(\strat1)$, 
\begin{displaymath}
\lim_{t \to \infty} (\eta_t \ast \delta_x)(\varphi) = \nu_\cM(\varphi)
\end{displaymath}

\item[{\rm (ii)}] (cf. Theorem~\ref{theorem:rw-uniformity})
Let $\cM$ be an affine invariant submanifold. 
Then for any $\varphi\in C_c(\strat1)$ and any $\epsilon>0$
there are affine invariant submanifolds $\cN_1,\ldots,\cN_\ell$
properly contained in $\cM$ such that 
for any compact subset $F\subset \cM \setminus(\cup_{j=1}^\ell \cN_j)$ 
there exists $T_0$ so that for all $T>T_0$ and any $x\in F$,
\begin{displaymath}
\left|(\eta_t \ast \delta_x) (\varphi) -\nu_{\cM} (\varphi)\right| < \epsilon. 
\end{displaymath}
\end{itemize}
\end{theorem}

\subsection{Equidistribution for sectors}
We define a sequence of probability measures $\vartheta_t$ on $SL(2,\reals)$ 
by
\begin{displaymath}
\vartheta_t(\varphi) = \frac{1}{t} \int_{0}^t \frac{1}{2\pi} \int_{0}^{2\pi}
\varphi(a_s r_\theta) \, d\theta \, ds.
\end{displaymath}
More generally, if $I \subset [0,2\pi]$ is an interval, then we define
\begin{displaymath}
\vartheta_{t,I}(\varphi) = \frac{1}{t} \int_{0}^t \frac{1}{|I|} \int_I
\varphi(a_s r_\theta) \, d\theta \, ds.
\end{displaymath}

We have the following: 
\begin{prop}
\label{prop:sector-avoid-singular-set}
Let $\cN$ be a (possibly empty) proper affine invariant submanifold. 
Then for any $\epsilon>0$, 
there exists an open set $\Omega_{\cN,\epsilon}$ containing $\cN$ with
$(\Omega_{\cN,\epsilon})^c$ compact such that for any
compact $F\subset\strat1\setminus\cN$ there exists $t_0 \in \reals$
so that for all $t>t_0$ and all $x\in F$, we have
\begin{displaymath}
(\vartheta_{t,I} \ast \delta_x)(\Omega_{\cN,\epsilon})<\epsilon.
\end{displaymath}
\end{prop}

\begin{proof} This proof is virtually identical to the proof of
  Proposition~\ref{prop:rw-avoid-singular-set}. 
It is enough to prove the statement for the case $I =
  [0,2\pi]$. Let $f_\cN$ 
be the function of Proposition~\ref{prop:main:proposition}.  Let $b >
0$ be as in Proposition~\ref{prop:main:proposition} (b), 
and let
\begin{displaymath}
\Omega_{\cN,\epsilon}=\left\{p: f_\cN( p)>(b+1)/\epsilon\right\}^0,
\end{displaymath} 
where $E^0$ denotes the interior of $E$. 

Suppose $F$ is a compact subset of $\strat1 \setminus \cN$. 
Let $m_F=\sup\{f_\cN(x): x\in F\}$. 
By Proposition~\ref{prop:main:proposition} (b) with $c = \frac{1}{2m_F}$, 
there exists $t_1 > 0$ such that 
\begin{displaymath}
(A_t f_\cN)(x)<\frac{1}{m_F} f_\cN(x)+ b\leq 1+
b, \qquad \text{for all $t > t_1$ and all $x\in F$.}
\end{displaymath}
By Proposition~\ref{prop:main:proposition} (a) there exists $R > 0$
such that $f_\cN(a_t x) \le R f_\cN(x)$ for $0 \le t \le t_1$. Now
choose $t_0$ so that $t_1 R/t_0  < m_F/2$. Then, for $t > t_0$, 
\begin{multline*}
(\vartheta_t \ast \delta_x)(f_\cN) = \frac{1}{t}\int_0^t (A_s
f_\cN)(x) \, ds 
= \frac{1}{t} \int_{0}^{t_1} (A_s f_\cN)(x) \, ds + \frac{1}{t}
\int_{t_1}^t (A_s f_\cN)(x) \, ds \\
\le \frac{t_1 R}{t} f_\cN(x) + (\frac{m_F}{2} f_\cN(x) + b) 
\le m_F f_\cN(x) + b \le 1+b. 
\end{multline*}
Thus for any $x\in F$, $t > t_0$ and $L>0$ we have 
\begin{equation}
\label{eq:eq:avoid2}
(\vartheta_{t} \ast \delta_x)(\{p: f_\cN( p)>L\})<(b+1)/L. 
\end{equation}
Then (\ref{eq:eq:avoid2}) implies that
$
(\vartheta_t \ast \delta_x)(\Omega_{\cN,\epsilon})<\epsilon.
$
Also, Proposition~\ref{prop:main:proposition} (a) 
implies that $\Omega_{\cN,\epsilon}$ is a neighborhood of $\cN$ and
\begin{displaymath}
(\Omega_{\cN,\epsilon})^c = \overline{\{ p \st f_\cN(p) \le
(b+1)/\epsilon \}} 
\end{displaymath}
is compact.
\end{proof}

\begin{lemma}
\label{lemma:sector:P:invariance} 
Suppose $t_i \to \infty$, $x_i \in \strat1$,   
and $\vartheta_{t_i,I} \ast \delta_{x_i} \to \nu$. Then
$\nu$ is invariant under $P$ (and then by
Theorem~\ref{theorem:P:measures} also invariant under
$SL(2,\reals)$). 
\end{lemma}

\begin{proof} 
Let $A$ denote the diagonal subgroup of $SL(2,\reals)$,
  and let $U = \begin{pmatrix} 1 & \ast \\ 0 & 1
  \end{pmatrix}$. 
From the definition it is clear that $\nu$ is $A$-invariant. 
We will show it is also $U$-invariant; indeed it suffices to show this for
$u_s=\begin{pmatrix} 1 & s \\ 0 & 1
  \end{pmatrix}$ with $0\leq s\leq 1.$  
  
First note that for any $0<\theta<\pi/2$ we have
\be\label{eq:uni-circle}
\mbox{$r_\theta = g_\theta\hh u_{\tan\theta},$ where
$\h g_\theta=\left(\begin{array}{cc}\cos\theta & 0\\ \sin\theta & 1/\cos\theta\end{array}\right).$}
\ee
Therefore, for all $\tau>0$ we have
$a_\tau g_\theta a_\tau^{-1}=\begin{pmatrix} \cos\theta & 0\\ e^{-2\tau}\sin\theta & 1/\cos\theta
  \end{pmatrix}.$
We have
\begin{equation}
\label{eq:matrix-multiplication}
a_\tau r_\theta=a_\tau g_\theta u_{\tan\theta}=
a_\tau g_\theta a_\tau^{-1}u_{e^{2\tau}\tan\theta}\; a_\tau. 
\end{equation}
Fix some $0 < s < 1$, and define $s_\tau$ by $e^{2\tau}\tan
s_{\tau}=s$. Then, (\ref{eq:matrix-multiplication}) becomes
\begin{equation}
\label{eq:matrix-mult2}
a_\tau r_{s_\tau} = (a_\tau g_{s_\tau} a_\tau^{-1}) u_s a_\tau.
\end{equation}
For any $\varphi\in C_c(\strat1)$ and all $x$ we have
\be\label{eq:tri-cricle-close}
\varphi(u_sa_\tau r_{\theta}x)-\varphi(a_\tau r_{\theta}x)= (\varphi(u_sa_\tau r_{\theta}x)-\varphi(a_\tau r_{\theta+s_{\tau}}x))+(\varphi(a_\tau r_{\theta+s_{\tau}}x)-\varphi(a_\tau r_{\theta}x)).
\ee
We compute the contribution from the two parentheses separately. Note that
terms in the first parenthesis are close to each other thanks to
(\ref{eq:matrix-mult2}) and the definition of $s_\tau.$
The contribution from the second is controlled as the 
integral over $I$ and a ``small''
translate of $I$ are close to each other. 

We carry out the computation here.  First note that $s_\tau\to0$ as
$\tau\to\infty.$ Furthermore, this and~\eqref{eq:uni-circle} imply
that $a_\tau g_{s_\tau}a_{\tau}^{-1}$ tends to the identity matrix as
$\tau\to\infty.$ Therefore, given $\epsilon>0$, thanks to
(\ref{eq:matrix-mult2}) and the uniform
continuity of $\varphi$ we have 
\begin{displaymath}
|\varphi(u_sa_\tau r_{\theta}x)-\varphi(a_\tau
r_{\theta+s_{\tau}}x)|\leq\epsilon
\end{displaymath}
for all large enough $\tau$ and all $x\in\strat1$. Thus, for large enough
$n$ (depending on $\epsilon$ and $\varphi$), we get
\begin{equation}
\label{eq:first-parenth}
\frac{1}{t_n}\int_0^{t_n}\frac{1}{|I|}\int_I |\varphi(u_sa_\tau
r_{\theta}x_n)-\varphi(a_\tau r_{\theta+s_{\tau}}x_n)|\, d\theta
\, d\tau\leq 2 \epsilon.  
\end{equation}
As for the second parentheses on the right
side of~\eqref{eq:tri-cricle-close}, we have
\begin{align*}
  & \left| \frac{1}{t_n}\int_0^{t_n}\frac{1}{|I|}\int_I (\varphi(a_\tau
    r_{\theta+s_{\tau}}x_n)-\varphi(a_\tau r_{\theta}x_n))\, d\theta
    \, d\tau\right|\leq\\
  &\le \frac{1}{t_n}\int_0^{t_n}\left|\frac{1}{|I|}\int_{I+s_{\tau}}
    \varphi(a_\tau
    r_{\theta}x_n)d\theta-\frac{1}{|I|}\int_I\varphi(a_\tau
    r_{\theta}x_n) \, d\theta\right|d\tau\leq
  \frac{C_\varphi}{t_n}\int_0^{t_n} s_\tau \, d\tau \\
& \leq\frac{C'_{\varphi}}{t_n}, \qquad\text{since $s_\tau =
  O(e^{-2\tau})$ and thus the integral converges.}
\end{align*}
This, together with ~\eqref{eq:first-parenth}
and~\eqref{eq:tri-cricle-close},  implies
$
|\nu(u_s\varphi)-\nu(\varphi)|\leq 2\epsilon;
$ 
the lemma follows.   
\end{proof}

Now in view of Proposition~\ref{prop:sector-avoid-singular-set} and
Lemma~\ref{lemma:sector:P:invariance},
Theorem~\ref{theorem:sector:closure} and
Theorem~\ref{theorem:sector:uniformity} hold by
Theorem~\ref{theorem:general:measures}.

\subsection{Equidistribution for some F{\o}lner sets}
In this subsection, we prove Theorem~\ref{theorem:folner:closure} and
Theorem~\ref{theorem:folner:uniformity}. These theorems can be easily derived
from Theorem~\ref{theorem:sector:closure} and
Theorem~\ref{theorem:sector:uniformity}, but we choose to derive them
directly from Theorem~\ref{theorem:general:measures}. 

Fix $r > 0$, and define a family of probability measures
$\lambda_{t,r}$ on $SL(2,\reals)$ by
\begin{displaymath}
\lambda_{t, r}(\varphi)=\frac{1}{rt}\int_0^t\int_0^r\varphi(a_\tau
u_s) \, ds \, d\tau.
\end{displaymath}

The supports of the measures $\lambda_{t,r}$ form a F{\o}lner family as
$t \to \infty$ (and $r$ is fixed). Thus, any limit measure of the
measures $\lambda_{t_i,r} \ast \delta_{x_i}$ is $P$-invariant (and
thus $SL(2,\reals)$-invariant by Theorem~\ref{theorem:P:measures}). 
Therefore it remains to prove:
\begin{prop}
\label{prop:folner-avoid-singular-set}
Let $\cN$ be a (possibly empty) proper affine invariant submanifold. 
Then for any $\epsilon>0$, 
there exists an open set $\Omega_{\cN,\epsilon}$ containing $\cN$ with
$(\Omega_{\cN,\epsilon})^c$ compact such that for any
compact $F\subset\strat1\setminus\cN$ there exists $t_0 \in \reals$
so that for all $t>t_0$ and all $x\in F$, we have
\begin{displaymath}
(\lambda_{t,r} \ast \delta_x)(\Omega_{\cN,\epsilon})<\epsilon.
\end{displaymath}
\end{prop}

\begin{proof}
It is enough to prove the statements for $r = \tan 0.01$. We may write as
in the proof of Lemma~\ref{lemma:sector:P:invariance}
\begin{displaymath}
r_\theta = g_\theta u_{\tan \theta}
\end{displaymath}
and thus
\begin{displaymath}
a_t u_{\tan \theta} = a_t g_\theta^{-1} r_\theta = (a_t g_\theta^{-1}
a_t^{-1}) a_t r_\theta
\end{displaymath}
Let $I = (0,0.01)$.
Note that $a_t g_\theta^{-1} a_t^{-1}$ remains bounded for $\theta \in
I$ as $t \to \infty$. Also, the derivative of $\tan \theta$ is bounded
between two non-zero constants for $\theta \in I$. Therefore, by
Proposition~\ref{prop:main:proposition} (c), for all $t$ and $x$, 
\begin{displaymath}
(\lambda_{t,r}  \ast \delta_x)(f_\cN) \le C (\vartheta_{t,I} \ast
\delta_x)(f_\cN),
\end{displaymath}
where $C$ depends only on the constant $\sigma$ in
Proposition~\ref{prop:main:proposition} (c). Therefore, for all $t$
and $x$, 
\begin{displaymath}
(\lambda_{t,r}  \ast \delta_x)(f_\cN) \le C' (\vartheta_{t} \ast
\delta_x)(f_\cN),
\end{displaymath}
where $C' = C/|I|$. Now let
\begin{displaymath}
\Omega_{\cN,\epsilon}=\left\{p: f_\cN( p)>C(b+1)/\epsilon\right\}^0.
\end{displaymath} 
The rest of the proof is exactly as in
Proposition~\ref{prop:sector-avoid-singular-set}. 
\end{proof}

Now Theorem~\ref{theorem:folner:closure} and
Theorem~\ref{theorem:folner:uniformity} follow from
Theorem~\ref{theorem:general:measures}.

\subsection{Proofs of Theorem~\ref{theorem:closure:submanifold}, 
  Theorem~\ref{theorem:closed:P:invariant:set} and
  Theorem~\ref{theorem:weak:asymptotics}}

\begin{proof}[Proof of Theorem~\ref{theorem:closure:submanifold}]
This is an immediate consequence of
Theorem~\ref{theorem:folner:closure}. 
\end{proof}

\begin{proof}[Proof of Theorem~\ref{theorem:closed:P:invariant:set}]
Suppose $\cA \subset \strat1$ is a closed $P$-invariant subset. Let
$Y$ denote the set of affine invariant manifolds contained in $\cA$,
and let $Z$ consist of the set of maximal elements of $Y$ (i.e. elements
of $Y$ which are not properly contained in another element of
$Y$). By Theorem~\ref{theorem:closure:submanifold}, 
\begin{displaymath}
\cA = \bigcup_{\cN \in Y} \cN = \bigcup_{\cN \in Z} \cN. 
\end{displaymath}
We now claim that $Z$ is finite. Suppose not, then there exists an
infinite sequence $\cN_n$ of distinct submanifolds in $Z$. Then by
Theorem~\ref{theorem:mozes-shah} there exists a subsequence
$\cN_{n_j}$ such that $\nu_{\cN_{n_j}} \to \nu_\cN$ where $\cN$ is another
affine invariant manifold which contains all but finitely many
$\cN_{n_j}$. Without loss of generality, we may assume that $\cN_{n_j}
\subset \cN$ for all $j$. 

Since $\nu_{\cN_{n_j}} \to \nu_{\cN}$, the union of the $\cN_{n_j}$ is
dense in $\cN$. Since $\cN_{n_j} \subset \cA$ and $\cA$ is closed,
$\cN \subset \cA$. Therefore $\cN \in Y$. But $\cN_{n_j} \subset \cN$,
therefore $\cN_{n_j} \not\in Z$. This is a contradiction. 
\end{proof}

\begin{proof}[Proof of Theorem~\ref{theorem:weak:asymptotics}]
This is a consequence of Theorem~\ref{theorem:sector:closure}; see
\cite[\S{3}-\S{5}]{EMas} for the details. See also
\cite[\S{8}]{Eskin:Marklof:Morris} 
for an axiomatic formulation and an outline of the argument. 

We note that since we do not have a convergence theorem for averages
of the form
\begin{displaymath}
\lim_{t \to \infty} \frac{1}{2\pi} \int_0^{2\pi} \varphi(a_t r_\theta
x) \, d\theta
\end{displaymath}
and therefore we do not know that e.g. assumption (C) of
\cite[Theorem~8.2]{Eskin:Marklof:Morris} is satisfied. But by
Theorem~\ref{theorem:sector:closure} we do have convergence for the
averages
\begin{displaymath}
\lim_{t \to \infty} \frac{1}{t} \int_0^t \frac{1}{2\pi} \int_0^{2\pi}
\varphi(a_s r_\theta x) \, d\theta \, ds. 
\end{displaymath}
Since we also have an extra average on the right-hand side of
Theorem~\ref{theorem:weak:asymptotics}, the proof goes through
virtually without modifications. 
\end{proof}

\section{Recurrence Properties}
\label{sec:recurrence}

Recall that for a function $f: \cH_1(\alpha) \to \reals$,
\begin{displaymath}
(A_t f)(x) = \frac{1}{2\pi} \int_0^{2\pi} f(a_t r_\theta x). 
\end{displaymath}
\begin{theorem}[\cite{EMas}, \cite{A}]
\label{theorem:alpha-fun}
  There exists a continuous, proper, $SO(2)$-invariant function
  $u:\strat1\to [2,\infty)$ such that
\begin{itemize}
\item[(i)] There exists $m \in \reals$ such that 
for all $x \in \cH_1(\alpha)$ and all $t > 0$, 
\begin{equation}
\label{eq:log-unif}
e^{-mt} u(x) \le u(a_t x) \le e^{m t} u(x)
\end{equation}
\item[(ii)] 
There exists constants $t_0
  > 0$, $\tilde{\eta} > 0$ and $\tilde{b} > 0$ such that
  for all $t\geq t_0$ and all $x\in\strat1$ we have
\begin{equation}
\label{eq:sup-harm-u}
A_t u(x)\leq \tilde{c}u(x)+\tilde{b}, \quad \text{ with $\tilde{c} =
  e^{-\tilde{\eta} t}$} 
\end{equation}
\end{itemize}
\end{theorem} 

We state some consequences of Theorem~\ref{theorem:alpha-fun},
{}{mostly} from
\cite{A}: \mc{give precise references}

\begin{theorem}
\label{theorem:all:measures:return}
For any $\rho > 0$, 
there exists a compact $K_\rho \subset \cH_1(\alpha)$ such that for any
$SL(2,\reals)$-invariant probability measure $\nu$, 
\begin{displaymath}
\nu(K_\rho) > 1-\rho.
\end{displaymath}
\end{theorem}

\begin{proof} 
{}{ 
The fact that this follows from Theorem~\ref{theorem:alpha-fun} is
well-known, and can be extracted e.g. from the proof of \cite[Lemma
2.2]{EMar}. For a 
self-contained argument one may use Lemma~\ref{lemma:in:L1} in the
present paper with $\sigma = e^{-m}$, $c = c_0(\sigma)$ and $t_0$
sufficiently large so that $e^{-\tilde{\eta}t_0} < c$, 
to obtain the estimate
\begin{displaymath}
\int_{\strat1} u(x) \, d\nu(x) < B,
\end{displaymath}
where $B$ depends only on the constants of
Theorem~\ref{theorem:alpha-fun}. This implies that 
\begin{displaymath}
\nu\left(\{ x \st u(x) > B/\rho \}\right) < \rho, 
\end{displaymath}
as required. 
}
\end{proof}

\begin{theorem}
\label{theorem:fast:return}
Let $K_\rho$ be as in
Theorem~\ref{theorem:all:measures:return}. Then, if $\rho > 0$ is
sufficiently small, there exists a constant $m'' > 0$ 
such that for all $x \in \cH_1(\alpha)$ there
exists $\theta \in [0,2\pi]$ and $\tau \le m'' \log u(x)$ such that
$x' \equiv a_\tau r_\theta x \in K_\rho$. 
\end{theorem}
\begin{proof}
{}{
This follows from \cite[Theorem~2.2]{A}, with $\delta=1/2$.
}
\end{proof}

\begin{theorem}
\label{theorem:exponential:return}
For  $x\in\strat1$ and a compact set $K_* \subset \strat1$
define
\begin{displaymath}
\ical_1(t) = \{\theta\in[0,2\pi]: |\{\tau\in[0,t]:a_\tau
  r_\theta x\in K_*\}|>t/2\},
\end{displaymath}
and
\begin{displaymath}
\ical_2(t)=[0,2\pi]\setminus\ical_1(t). 
\end{displaymath}

Then, 
 there exists some $\eta_1>0,$ a compact
subset $K_*,$ and constants $L_0 > 0$ and $\eta_0 >
  0$ such that for any $t > 0$, 
\begin{equation}
\label{eq:exp-recurrence} 
\text{if $\log u(x) < L_0 + \eta_0 t$,} \qquad  
 \text{then $|\ical_2( t)|<e^{-\eta_1 t}.$} 
\end{equation}
\end{theorem}
Theorem~\ref{theorem:exponential:return} is not formally stated in
\cite{A}, but is a combination of \cite[Theorem 2.2]{A} and
\cite[Theorem 2.3]{A}. (In the proof of \cite[Theorem 2.3]{A}, one should
use \cite[Theorem 2.2]{A} to control the distribution of $\tau_0$).

\section{Period Coordinates and the Kontsevich-Zorich cocycle}
\label{sec:period}
Let $\Sigma \subset M$ denote the set of zeroes of $\omega$. Let
$\{ \gamma_1, \dots, \gamma_k\}$ denote a $\zed$-basis for
the relative 
homology group $H_1(M,\Sigma, \zed)$. {}{(It is
  convenient to assume that the basis is obtained by extending a
  symplectic basis for the absolute homology group $H_1(M,\zed)$.)}
We can define a map $\Phi:
\cH(\alpha) \to \cx^k$ by 
\begin{displaymath}
\Phi(M,\omega) = \left( \int_{\gamma_1} \omega, \dots, \int_{\gamma_k}
  w \right)
\end{displaymath}
The map $\Phi$ (which depends on a choice of the basis $\{ \gamma_1,
\dots, \gamma_n\}$) is a local coordinate system on $(M,\omega)$.  
Alternatively,
we may think of the cohomology class $[\omega] \in H^1(M,\Sigma, \cx)$ as a
local coordinate on the stratum $\cH(\alpha)$. We
will call these coordinates {\em period coordinates}. 

\bold{The $SL(2,\reals)$-action and the Kontsevich-Zorich cocycle.}
We write $\Phi(M,\omega)$ as a $2 \cross n$ matrix $x$. The action of $g =
\begin{pmatrix} a & b \\ c & d \end{pmatrix} \in
SL(2,\reals)$ in these coordinates is linear. 
{}{Let $\operatorname{Mod}(M,\Sigma)$ be
the mapping class group of $M$ fixing each zero of $\omega$.} 
We choose some
fundamental domain for the action of 
{}{$\operatorname{Mod}(M,\Sigma)$}, and
think of the dynamics on the fundamental domain. Then, the
$SL(2,\reals)$ action becomes
\begin{displaymath}
x = \begin{pmatrix} x_1 & \dots & x_n \\ y_1 & \dots & y_n \end{pmatrix}
\to gx = \begin{pmatrix} a & b \\ c & d \end{pmatrix} \begin{pmatrix} x_1 & \dots & x_n \\ y_1 & \dots & y_n
\end{pmatrix} A(g,x),
\end{displaymath}
where $A(g,x) \in \operatorname{Sp}(2g,\zed) \ltimes \reals^{\noz-1}$
is the {\em Kontsevich-Zorich 
cocycle}. Thus, $A(g,x)$ is change of basis one needs to perform to return the
point $gx$ to the fundamental domain. It can be interpreted as the
monodromy of the Gauss-Manin connection (restricted to the orbit of
$SL(2,\reals)$).

\section{The Hodge norm}
\label{sec:hodge}
Let $M$ be a Riemann surface. By definition, $M$ has a complex
structure. Let $\cH_M$ denote the set of holomorphic $1$-forms on
$M$. One can define {\em Hodge inner product} on $\cH_M$ by
\begin{displaymath}
\langle \omega, \eta \rangle = \frac{i}{2} \int_M \omega \wedge \bar{\eta}.
\end{displaymath}
We have a natural map $r: H^1(M,\reals) \to \cH_M$ which sends
a cohomology class $\lambda \in H^1(M,\reals)$ to the 
holomorphic $1$-form $r(\lambda) \in \cH_M$ such that 
the real part of $r(\lambda)$ (which is a harmonic $1$-form)
represents $\lambda$. We can thus define  the Hodge inner product on
$H^1(M,\reals)$ by $\langle \lambda_1, \lambda_2 \rangle = \langle
r(\lambda_1), r(\lambda_2) \rangle$. We have
\begin{displaymath}
\langle \lambda_1, \lambda_2 \rangle = \int_M \lambda_1 \wedge *\lambda_2,
\end{displaymath}
where $*$ denotes the Hodge star operator, and we choose harmonic representatives of $\lambda_1$ and $*\lambda_2$ to evaluate the integral.
We denote the associated norm by $\| \cdot \|_M$. This is the {\em Hodge
  norm}, see \cite{FarkasKra}.

If $x = (M,\omega) \in \cH_1(\alpha)$, we will often write $\| \cdot
\|_x$ to denote the Hodge norm $\| \cdot \|_M$ on
$\hr$. Since $\| \cdot \|_x$ depends only on $M$, we have $\|
\lambda \|_{kx} = \|\lambda \|_x$ for all $\lambda \in \hr$ and all $k
\in SO(2)$.

Let $E(x)=\mbox{span}\{\Rfrak(\omega),\Ifrak(\omega)\}.$
{}{(Many authors refer to $E(x)$ as the ``standard
  space'').}
We let $p: \rhr\rightarrow\hr$ denote the natural projection; using this map
$p(E(x))\subset\hr.$ 
For any $v \in E(x)$ and any point $y$ in the $SL(2,\reals)$ orbit of
$x$, the Hodge norm $\|v\|_y$ of $v$ at $y$ can be explicitly
computed. In fact, the following elementary lemma holds:
\begin{lemma}
\label{lemma:Hodge:norm:in:Ex}
Suppose $x \in \strat1$, $g = \begin{pmatrix}a_{11} & a_{12} \\ a_{21}
  & a_{22} \end{pmatrix}
\in SL(2,\reals)$, 
\begin{displaymath}
v = v_1 p(\Rfrak(\omega)) + v_2 p(\Ifrak(\omega)) \in p(E(x)). 
\end{displaymath}
Let
\begin{equation}
\label{eq:def:u1:u2}
\begin{array}[t]({cc}) u_1 & u_2 \end{array} =  
\begin{array}[t]({cc}) v_1 & v_2 \end{array} 
\begin{array}[t]({cc}) a_{11} & a_{12} \\ a_{21} & a_{22} \end{array}^{-1} 
\end{equation}
Then, 
\begin{equation}
\label{eq:Hodge:v:gx}
\| v \|_{g x} = \|u_1^2 + u_2^2\|^{1/2}. 
\end{equation}
\end{lemma}
\begin{proof}
Let 
\begin{equation}
\label{eq:def:c1:c2}
c_1 = a_{11} p(\Rfrak(\omega)) + a_{12} p(\Ifrak(\omega)) \qquad c_2 = a_{21} p(\Rfrak(\omega)) + a_{22}p(\Ifrak(\omega)).
\end{equation}
By the definition of the $SL(2,\reals)$ action, $c_1 + i c_2$ 
is holomorphic on $gx$. Therefore, by the definition of the Hodge
star operator, at $gx$, 
\begin{displaymath}
\ast c_1 = c_2, \qquad \ast c_2 = -c_1. 
\end{displaymath}
Therefore, 
\begin{displaymath}
\|c_1\|_{gx}^2 = c_1 \wedge \ast c_1 = c_1 \wedge c_2 = (\det g)
\Rfrak(\omega)) \wedge p(\Ifrak(\omega)) = 1,
\end{displaymath}
where for the last equality we used the fact that $x \in \strat1$. 
Similarly, we get
\begin{equation}
\label{eq:Hodge:ci:gx}
\|c_1\|_{gx} = 1, \qquad \|c_2\|_{gx} = 1, \qquad \langle c_1, c_2
\rangle_{gx} = 0.
\end{equation}
Write
\begin{displaymath}
v = v_1 p(\Rfrak(\omega)) + v_2 p(\Ifrak(\omega)) = u_1 c_1 + u_2 c_2. 
\end{displaymath}
Then, in view of (\ref{eq:def:c1:c2}), $u_1$ and $u_2$ are given by
(\ref{eq:def:u1:u2}). The equation (\ref{eq:Hodge:v:gx}) follows from
(\ref{eq:Hodge:ci:gx}). 
\end{proof}

On the complementary subspace to $p(E(x))$ there is no explicit formula
comparable to Lemma~\ref{lemma:Hodge:norm:in:Ex}. However, we have the
following fundamental result due to Forni~\cite[Corollary 2.1]{Forni}, see also
\cite[Corollary~2.1]{Forni:Matheus:Zorich}:
\begin{lemma}
\label{lemma:forni}
\label{lemma:furstenberg}
\label{l;unif-hyp}
There exists a continuous function
$\Lambda:\hcal_1(\alpha)\rightarrow(0,1)$ such that; for any $c \in
\hr$ with $c\wedge p(E(x))=0$, any $x \in \strat1$ and any $t>0$ we
have
\begin{displaymath}
\|c\|_x e^{-\beta_t(x)} \le 
\|c\|_{a_t x} \le \|c\|_x e^{\beta_t(x)}
\end{displaymath}
where $\beta_t( x)=\int_0^t\Lambda(a_\tau x)\, d\tau$.
\end{lemma}

Let $\cI_1(t)$ and $\cI_2(t)$ be as in
Theorem~\ref{theorem:exponential:return}.  Now compactness of
$K_*$ and Lemma~\ref{l;unif-hyp} imply that:
\begin{multline}
\label{eq:second-exp}
\text{there exists $\eta_2>0$ such that {}{for all $x
    \in \strat1$},} \\ \text{if   $t>t_0$ and $\theta\in\ical_1(t)$, then
  $\beta_t(r_\theta x)<(1-\eta_2)t.$}
\end{multline}


\section{Expansion on average of the Hodge norm.}
\label{sec;furstenberg}

Recall that $p: \rhr\rightarrow\hr$ denotes the natural projection. 
Let
$\cM_1$ be an affine invariant suborbifold of $\hcal_1(\alpha)$ and
let $\cM=\bbr\cM_1$ be as above. Then $\cM$ is given by complex
linear equations in period coordinates and is $GL(2,\reals)$-invariant. 
We let $L$ denote this subspace in $\rhr.$

Recall that $\hr$ is endowed with a natural symplectic structure given
by the wedge product on de Rham cohomology and also the Hodge inner
product. It is shown in \cite{Avila:Eskin:Moeller:yeti} that the wedge
product 
restricted to $p(L)$ is non-degenerate. Therefore, there exists an
$SL(2,\reals)$-invariant complement for $p(L)$ in $\hr$ which we
denote by $p(L)^\perp$.

We will use the following elementary lemma with $d=2,3$:
\begin{lemma}
\label{lemma:EMM51}
Let $V$ be a $d$-dimensional vector space on which $SL(2,\reals)$
acts irreducibly, and let $\| \cdot \|$ be any $SO(2)$-invariant norm
on $V$. Then there exists $\delta_0(d) > 0$ (depending on $d$), such that
for any $\delta < \delta_0(d)$ any $t > 0$ and any $v \in V$, 
\begin{displaymath}
\frac{1}{2\pi} \int_0^{2\pi} \frac{d\theta}{\|a_t r_\theta v\|^\delta}
\le \frac{e^{-k_d t}}{\|v\|^\delta},
\end{displaymath}
where $k_d = k_d(\delta) > 0$. 
\end{lemma}
\begin{proof} 
{}{This is essentially the case $G=SL(2,\reals)$ of
  \cite[Lemma~5.1]{EMM1}. The exponential estimate in the
  right-hand-side is not stated in
  \cite[Lemma~5.1]{EMM1} but follows easily from the proof of the lemma.}
\end{proof}

\bold{The space $H'(x)$ and the function $\psi_x$.}
For $x = (M,\omega)$, let
\begin{displaymath}
H'(x) = \{ v \in H^1(M,\cx) \st v \wedge \overline{p(\omega)} + p(\omega)
\wedge \overline{v} = 0 \}. 
\end{displaymath}
We have, for any $x = (M,\omega)$, 
\begin{displaymath}
H^1(M,\cx) = \reals \, p(\omega) \oplus H'(x). 
\end{displaymath}
{}{(Here and below, we are considering $H^1(M,\cx)$ as a
  real vector space.)}
For $v \in H^1(M,\cx)$, let
\begin{displaymath}
\psi_x(v) = \frac{\|v\|_x}{\|v'\|_x} \quad\text{where $v = \lambda \, 
  p(\omega) + v'$, $\lambda 
\in \reals$, $v' \in H'(x)$.} 
\end{displaymath}
Then $\psi_x(v)\ge 1$, and $\psi_x(v)$ 
 is bounded if $v$ is bounded away from $\reals \,
p(\omega)$. 

\subsection{Absolute Cohomology}
\label{sec:subsec:absolute}
Fix some $\delta \le 0.1 \min(\eta_1,\eta_2,
\delta_0(2),\delta_0(3))$. 
For $g = \begin{pmatrix} a & b \\ c & d \end{pmatrix}$ and $v \in
H^1(M,\cx)$, we write
\begin{equation}
\label{eq:action:SL2}
g v = a \, \Rfrak(v) + b \, \Ifrak(v) + i( c \, \Rfrak(v) + d \, \Ifrak(v)). 
\end{equation}
\begin{lemma}
\label{lemma:ineq:absolute:thick:part}
There exists {}{$C_0 > 1$} such that
for all $x = (M,\omega) \in \cH_1(\alpha)$, all $t > 0$  and
all $v \in H^1(M,\cx)$ we have
\begin{equation}
\label{eq:lemma:ineq:absolute}
\frac{1}{2\pi} \int_0^{2\pi} \frac{d\theta}{(\|a_t r_\theta v
  \|_{a_t r_\theta x})^{\delta/2}} \le
\min\left(\frac{C_0}{\|v\|_x^{\delta/2}}, \frac{\psi_x(v)^{\delta/2}
  \kappa(x,t)}{\|v\|_x^{\delta/2}}\right), 
\end{equation}
where 
\begin{itemize}
\item[{\rm (a)}] $\kappa(x,t) \le C_0$ for all
  $x$ and all $t$, and
\item[{\rm (b)}]
There exists $\eta > 0$ such that 
\begin{displaymath}
\kappa(x,t) \le {}{C_0} e^{-\eta t}, 
\qquad \text{ provided $\log u(x) < L_0 + \eta_0 t$ }. 
\end{displaymath}
{}{where the constants $L_0$ and $\eta_0$ are as in Theorem~\ref{theorem:exponential:return}.}
\end{itemize}
\end{lemma}

\begin{proof}
For $x=(M,\omega)\in\hcal_1(\alpha)$ we have an $\sl$-invariant and
Hodge-orthogonal decomposition 
\begin{displaymath}
H^1(M,\reals) = p(E(x)) \dirsum H^1(M,\reals)^\perp, 
\end{displaymath}
where $E(x)=\mbox{span}\{\Rfrak(\omega),\Ifrak(\omega)\}$ and
$H^1(M,\reals)^\perp(x)=\{c\in\hr: c\wedge p(E(x)) =0\}.$
For a subspace $V \subset \hr$, let $V_\cx \subset H^1(M,\cx)$ 
denote its complexification. Then, we have
\begin{equation}
\label{eq:complex:hodge:decomp}
H^1(M,\cx) = p(E(x))_\cx {}{\dirsum} H^1(M,\reals)^\perp_\cx(x).
\end{equation}
Note that $H^1(M,\reals)^\perp_\cx(x) \subset H'(x)$. 
We can write
\begin{displaymath}
v = \lambda\omega + u + w,
\end{displaymath}
where $\lambda \in \reals$, 
$u \in p(E(x))_\cx \cap H'(x)$, $w \in H^1(M,\reals)^\perp_\cx(x)$. 
Since $u \in p(E(x))_\cx$, we may write
\begin{displaymath}
u = u_{11} \, p(\Rfrak(\omega)) + u_{12} \, p(\Ifrak(\omega)) + i(
u_{21} \, p(\Rfrak(\omega)) + u_{22} \, p(\Ifrak(\omega))). 
\end{displaymath}
Since $u \in H'(x)$, 
\begin{equation}
\label{eq:traceless}
u_{11} + u_{22} = 0. 
\end{equation}
Recall that the Hilbert-Schmidt norm $\| \cdot \|_{HS}$ of a 
matrix is the square root of the sum of the squares of the entries.
Then, 
\begin{align}
\label{eq:conjugation:action}
(\| a_t r_\theta (p(\omega) + u)\|_{a_t r_\theta x})^2 & = \left\| (a_t
  r_\theta) \begin{pmatrix} 
  \lambda+ u_{11} & u_{12} \\  u_{21} & \lambda+ u_{22} \end{pmatrix} (a_t
r_\theta)^{-1} \right\|^2_{HS} && \text{by
Lemma~\ref{lemma:Hodge:norm:in:Ex} and (\ref{eq:action:SL2})} \notag \\
& = \lambda^2 + 
\left\| (a_t
  r_\theta) \begin{pmatrix} 
  u_{11} & u_{12} \\  u_{21} & u_{22} \end{pmatrix} (a_t
r_\theta)^{-1} \right\|^2_{HS} && \text{by (\ref{eq:traceless})}.  
\end{align}
Since the decomposition (\ref{eq:complex:hodge:decomp}) 
is Hodge orthogonal, it follows that for
all $t$ and all $\theta$, 
\begin{equation}
\label{eq:threeway:decomp}
(\|a_t r_\theta v\|_{a_t r_\theta x})^2 = \lambda^2 + (\|a_t r_\theta
u\|_{a_t r_\theta x})^2 + (\|a_t r_\theta w \|_{a_t r_\theta x})^2. 
\end{equation}
By (\ref{eq:conjugation:action}), (\ref{eq:traceless}) and
Lemma~\ref{lemma:EMM51},  
\begin{equation}
\label{eq:bound:sl2:part}
\frac{1}{2\pi}\int_0^{2\pi}\frac{d\theta}{(\|a_t r_\theta
  u\|_{a_t r_\theta x})^{\delta/2}} \leq 
e^{-k_3 t}{\|u\|_x^{\delta/2}},
\end{equation}
where $k_3 > 0$. 
We now claim that
\begin{equation}
\label{eq:bound:otrhogonal:part}
\frac{1}{2\pi}\int_0^{2\pi}\frac{d\theta}{\|a_t r_\theta
  w\|^{\delta/2}} \leq \frac{\kappa_2(x,t)}{\|w\|^{\delta/2}}, 
\end{equation}
where for some absolute constant $C > 0$ {}{and $\eta >
  0$, and for $L_0$ and
  $\eta_0$ as in Theorem~\ref{theorem:exponential:return}, we have}
\begin{equation}
\label{eq:kappa2:strong:estimate}
\begin{cases}
\kappa_2(x,t) \le C  & \text{for all} \quad x \in \cH_1(\alpha), t \ge
0 \\
\kappa_2(x,t) \le C e^{-\eta t}
\quad & \text{provided $\log u(x) < L_0 + \eta_0 t$.}  
\end{cases}
\end{equation}

Assuming (\ref{eq:bound:otrhogonal:part}) and
  (\ref{eq:kappa2:strong:estimate}), we have 
\begin{align*}
& \frac{1}{2\pi} \int_0^{2\pi} \frac{d\theta}{(\|a_t r_\theta
  v\|_{a_t r_\theta x})^{\delta/2}} \\ 
& \le \frac{1}{2\pi} \int_0^{2\pi} \min
\left(\frac{1}{\lambda^{\delta/2}}, \frac{1}{(\|a_t r_\theta 
  u\|_{a_t r_\theta x})^{\delta/2}}, \frac{1}{(\|a_t r_\theta
  w\|_{a_t r_\theta x})^{\delta/2}} \right) \, d\theta  && \text{ by
(\ref{eq:threeway:decomp})} \\ 
 & \le \min\left( \frac{1}{\lambda^{\delta/2}}, \frac{1}{2\pi} \int_0^{2\pi}
   \frac{d\theta}{(\|a_t r_\theta 
  u\|_{a_t r_\theta x})^{\delta/2}}, \frac{1}{2\pi} \int_0^{2\pi} \frac{d\theta}{(\|a_t r_\theta
  w\|_{a_t r_\theta x})^{\delta/2}} \right)
 && \\
& \le \min\left(\frac{1}{\lambda^{\delta/2}},
  \frac{e^{-k_3 t}}{\|u\|_x^{\delta/2}},  
\frac{\kappa_2(x,t)}{\|w\|_x^{\delta/2}} \right) && \text{by
(\ref{eq:bound:sl2:part}) 
and (\ref{eq:bound:otrhogonal:part}).}
\end{align*}
Since we must have 
either $\lambda > \|v\|_x/3$, or $\|u\|_x >
  \|v\|_x/3$ or $\|w\|_x > \|v\|_x/3$,
we have for all $x$, $t$,
\begin{displaymath}
\min\left(\frac{1}{\lambda^{\delta/2}},
  \frac{e^{-k_3 t}}{\|u\|_x^{\delta/2}},  
\frac{\kappa_2(x,t)}{\|w\|_x^{\delta/2}} \right) \le \frac{3^{\delta/2}
\max(1,e^{-k_3 t}, \kappa_2(x,t))}{\|v\|_x^{\delta/2}} \le
\frac{C_0}{\|v\|_x^{\delta/2}}.  
\end{displaymath}
where for the last estimate we used the fact that both $k_3$ and
$\kappa_2$ are bounded functions. 
Also, we have $\|u + w \|_x = \psi_x(v)^{-1} \| v\|_x$, hence either $\|u\|_x \ge
\psi_x(v)^{-1} \|v\|_x/2$ or $\|w\|_x \ge \psi_x(v)^{-1} \|v\|_x/2$, 
and therefore, for all
$x$, $t$, 
\begin{displaymath}
\min\left(\frac{1}{\lambda^{\delta/2}},
  \frac{e^{-k_3 t}}{\|u\|_x^{\delta/2}},  
\frac{\kappa_2(x,t)}{\|w\|_x^{\delta/2}} \right) 
\le
\frac{\psi_x(v)^{\delta/2} \max(e^{-k_3 t},
  \kappa_2(x,t))}{\|v\|_x^{\delta/2}} \equiv
\frac{\psi_x(v)^{\delta/2} \kappa(x,t)}{\|v\|_x^{\delta/2}}.
\end{displaymath}
Therefore, (\ref{eq:lemma:ineq:absolute}) holds. 
This completes the proof of the lemma,
assuming (\ref{eq:bound:otrhogonal:part}) and
(\ref{eq:kappa2:strong:estimate}).

It remains to prove (\ref{eq:bound:otrhogonal:part}) and
(\ref{eq:kappa2:strong:estimate}). 
Let $L_0$ and $\eta_0$ be as in
Theorem~\ref{theorem:exponential:return}, and suppose $\log u(x) < L_0
+ \eta_0 t$. {}{Recall that $\cI_1(t)$ and $\cI_2(t)$
  are defined   relative to the compact set $K_*$ in
  Theorem~\ref{theorem:exponential:return}.} 
We have
\begin{displaymath}
\int_0^{2\pi}\frac{d\theta}{(\|a_t r_\theta w\|_{a_t r_\theta x})^{\delta/2}}
=\int_{\ical_1(t)}\frac{d\theta}{(\|a_t r_\theta w\|_{a_t r_\theta
    x})^{\delta/2}} 
+\int_{\ical_2(t)}\frac{d\theta}{(\|a_t r_\theta w\|_{a_t r_\theta
    x})^{\delta/2}}.  
\end{displaymath}
Using~\eqref{eq:exp-recurrence} and
Lemma~\ref{lemma:furstenberg} we get \mc{(more details?)}
\begin{displaymath}
\int_{\ical_2(
  t)}\frac{d\theta}{(\|a_t r_\theta w\|_{a_t r_\theta x})^{\delta/2}}\leq
\frac{e^{-\eta_1 t}e^{\delta t/2}}{\|v\|^{\delta/2}}.
\end{displaymath}
Also,
\begin{align*}
\int_{\ical_1(t)}\frac{d\theta}{(\|a_t r_\theta w\|_{a_t r_\theta
    x})^{\delta/2}} & \leq
\int_{\ical_1(t)}\frac{d\theta}{(\|\Rfrak(a_t r_\theta
  w)\|_{a_t r_\theta x})^{\delta/2}} && \text{since $\|z\| \ge \|\Rfrak(z)\|$} \\
& = \int_{\ical_1(t)}\frac{d\theta}{(\|e^t \Rfrak(r_\theta
  w)\|_{a_t r_\theta x})^{\delta/2}} && \text{ by (\ref{eq:action:SL2})} \\
& \le \int_{\ical_1(t)}\frac{e^{-(1-\beta_t(r_\theta x)) \delta
      t/2}} {\|\Rfrak(r_\theta w )\|_x^{\delta/2}} && \text{ by
    Lemma~\ref{lemma:furstenberg} }\\ 
& \leq
e^{-\eta_2 \delta t/2} \int_0^{2\pi}\frac{d\theta}{\|\Rfrak(r_\theta
  w)\|_x^{\delta/2}} && \text{ by (\ref{eq:second-exp}) }\\ 
& = e^{-\eta_2 \delta t/2} \int_0^{2\pi} \frac{d\theta}{\| \cos \theta
  \, \Rfrak(w) + \sin \theta \, \Ifrak(w) \|_x^{\delta/2}} && \\
& \leq \frac{C_2 e^{-\eta_2\delta
    t/2}}{\|w\|^{\delta/2}}. && \hspace{-0.75in}\text{since the integral
  converges.} 
\end{align*}
These estimates imply (\ref{eq:bound:otrhogonal:part})
and (\ref{eq:kappa2:strong:estimate}) for the case
when $\log u(x) < L_0 + \eta_0 t$.  
If $x$ is arbitrary, we need to show
  (\ref{eq:bound:otrhogonal:part}) holds with $\kappa_2(x,t) \le C$.
Note that 
\begin{align*}
\|a_t r_\theta w\|_{a_t r_\theta x} & \ge \| \Rfrak(a_t r_\theta w)
\|_{a_t r_\theta x} && \text{ since $\|z\| \ge \|\Rfrak(z)\|$} \\
& =  \|e^{t} (\cos \theta \, \Rfrak(w) + \sin \theta \,\Ifrak(w))\|_{a_t
  r_\theta x} && \text{by (\ref{eq:action:SL2})} \\
& = e^{t} \|\cos \theta \, \Rfrak(w) + \sin \theta \,\Ifrak(w)\|_{a_t
  r_\theta x} \\
& \ge \| \cos \theta \, \Rfrak(w) + \sin \theta \,\Ifrak(w) \|_x &&
\text{by Lemma~\ref{lemma:forni}.}
\end{align*}
Therefore, 
\begin{align*}
\frac{1}{2\pi} \int_0^{2\pi} \frac{d\theta}{(\|a_t r_\theta w\|_{a_t
    r_\theta x})^{\delta/2}}  & \le \frac{1}{2\pi} \int_0^{2\pi}
\frac{d\theta}{ \| \cos \theta \, \Rfrak(w) + \sin \theta \,\Ifrak(w)
  \|_x^{\delta/2}} && \\
& \le \frac{C_2}{\|w\|_x^{\delta/2}} && \hspace{-2in} \text{since
  {}{$\delta \leq 0.1$} and the
  integral converges.} 
\end{align*}
This completes the proof of (\ref{eq:bound:otrhogonal:part})
and (\ref{eq:kappa2:strong:estimate}) 
for arbitrary $x$. 
\end{proof}

\subsection{The Modified Hodge Norm}
For the application in \S\ref{sec:subsec:relative}, we will need to
consider a modification of the Hodge norm in the thin part of moduli
space.

\bold{The classes $c_\alpha$, $\ast c_\alpha$.}
Let $\alpha$ be a homology class in $H_1(M,\reals)$. We can define the
cohomology class $*c_{\alpha} \in H^1(M,\reals)$ so that for all
$\omega \in H^1(M, \reals)$, 
\begin{displaymath}
\int_\alpha \omega = \int_M \omega \wedge *c_\alpha. 
\end{displaymath}
Then, 
\begin{displaymath}
\int_M *c_\alpha \wedge *c_\beta = I(\alpha,\beta),
\end{displaymath}
where $I(\cdot, \cdot)$ denotes algebraic intersection number.
Let $\ast$ denote the Hodge star operator, and let 
\begin{displaymath}
c_\alpha = \ast^{-1}(*c_\alpha). 
\end{displaymath}
Then, 
we have, for any $\omega \in H^1(M,\reals)$, 
\begin{displaymath}
\langle \omega, c_\alpha \rangle = \int_M \omega \wedge *c_\alpha = 
\int_\alpha \omega,
\end{displaymath}
where $\langle \cdot, \cdot \rangle$ is the Hodge inner product. 
We note that $*c_\alpha$ is a purely topological construction which 
depends only on $\alpha$, but $c_\alpha$
depends also on the complex structure of $M$.

Fix $\epsilon_* > 0$ (the \emph{Margulis constant}) so that any two
curves of hyperbolic length less than $\epsilon_*$ must be disjoint. 

{}{
Let $\sigma$ denote the hyperbolic metric in the conformal
class of $M$. For a closed curve $\alpha$ on $M$, $\ell_\alpha(\sigma)$
denotes the length of the geodesic representative of 
$\alpha$ in the metric $\sigma$.} 

We recall the following: 
\begin{theorem}\cite[Theorem 3.1]{ABEM}
\label{theorem:hodge:hyperbolic}
For any constant $D > 1$ there exists a constant $c > 1$ such that
for any simple closed curve
$\alpha$ with $\ell_\alpha(\sigma) < D$, 
\begin{equation}
\label{eq:hodge:hyperbolic}
\frac{1}{c} \ell_\alpha(\sigma)^{1/2} \le \| c_\alpha \| <
c \, \ell_\alpha(\sigma)^{1/2}.
\end{equation}
Furthermore, if $\ell_\alpha(\sigma) < \epsilon_*$ 
and $\beta$ is the
shortest simple closed curve crossing $\alpha$, then
\begin{displaymath}
\frac{1}{c} \ell_\alpha(\sigma)^{-1/2} \le \| c_\beta \| <
c \, \ell_\alpha(\sigma)^{-1/2}.
\end{displaymath}
\end{theorem}

\bold{Short bases.} Suppose $(M,\omega) \in \strat1$. Fix $\epsilon_1
< \epsilon_*$ 
and let $\alpha_1, \dots, \alpha_k$ be the curves with
hyperbolic length less than $\epsilon_1$
on $M$. For $1 \le i \le k$,
let $\beta_i$ be the shortest curve in the flat metric defined by
$\omega$ with $i(\alpha_i, \beta_i) =1$. We can pick simple closed
curves $\gamma_r$, $1 \le r \le 2g-2k$ on $M$ so that the hyperbolic
length of each $\gamma_r$ is bounded by a constant $L$ depending only
on the genus, and so that the $\alpha_j$, $\beta_j$ and $\gamma_j$ are
a symplectic basis $\cS$ for $H_1(M,\reals)$. We will call such a
basis {\em short.} {}{A short basis is not unique, and
in the following we fix some measurable choice of a short basis at
each point of $\cH_1(\alpha)$. }

We now define a modification of the Hodge norm, which is similar to
the one used in \cite{ABEM}. 
The modified norm is defined on the tangent space to the
space of pairs $(M,\omega)$ where $M$ is a Riemann surface and
$\omega$ is a holomorphic $1$-form on $M$. Unlike the Hodge norm, 
the modified Hodge norm will
depend not only on the complex structure on $M$ but also on the choice
of a holomorphic $1$-form $\omega$ on $M$. Let $\{\alpha_i, \beta_i,
\gamma_r\}_{1 \le i \le k, 1 \le r \le 2g-2k}$ 
be a short basis for {}{$x=(M,\omega)$.}

We can write any $\theta \in H^1(M,\reals)$ as 
\begin{equation}
\label{eq:expand:in:basis}
\theta = \sum_{i=1}^k a_i (*c_{\alpha_i}) + \sum_{i=1}^k b_i
\ell_{\alpha_i}(\sigma)^{1/2} (*c_{\beta_i})  + \sum_{r=1}^{2g -2k }
u_i (*c_{\gamma_r}),
\end{equation}
We then define
\begin{equation}
\label{eq:def:modified:hodge:norm}
 \|\theta\|_x'' = \|\theta\|_x +  \left( \sum_{i=1}^k |a_i| +
   \sum_{i=1}^k |b_i| +
    \sum_{r=1}^{2g -2k} |u_r| \right).
\end{equation}
{}{We note that $\| \cdot \|''$ depends on the choice of
  short basis; however switching to a different short basis can change
  $\| \cdot \|''$ by at most a fixed multiplicative constant depending
  only on the genus. To manage this, we use the notation $A \approx B$
  to denote the fact that $A/B$ is bounded from above and below by
  constants depending on the genus.
}

From (\ref{eq:def:modified:hodge:norm}) we have for $1 \le i \le k$, 
\begin{equation}
\label{eq:star:c:alpha:prime:norm}
\|*\!c_{\alpha_i}\|_x'' \approx 1, 
\end{equation}
Similarly, {}
we have 
\begin{equation}
\label{eq:star:c:beta:prime:norm}
\|*\!c_{\beta_i}\|_x'' \approx  \| *\!c_{\beta_i} \|_x \approx
\frac{1}{\ell_{\alpha_i}(\sigma)^{1/2}}. 
\end{equation}
In addition, in view of Theorem~\ref{theorem:hodge:hyperbolic}, 
if $\gamma$ is any other moderate length curve on $M$, 
$\|*\!c_\gamma\|_x'' \approx \|*\!c_\gamma \|_x = O(1)$. 
Thus, if $\cB$ is a short basis at {}{$x=(M,\omega)$}, then 
for any $\gamma \in \cB$, 
\begin{equation}
\label{eq:short:basis:extremal:length}
\Ext_\gamma({}{x})^{1/2} \approx \|\!*\!c_\gamma\| \le
\|\!*\!c_\gamma\|'' \end{equation}
(By $\Ext_\gamma({}{x})$ we mean the
extremal length of $\gamma$ in $M$, {}{where $x =
  (M,\omega)$}.)

\bold{Remark.} From the construction, we see that the modified Hodge
norm is greater than the Hodge norm. Also, if the flat length of 
shortest curve in
the flat metric defined by $\omega$ is greater than $\epsilon_1$, then
for any cohomology class $\lambda$, for some $C$ depending on
$\epsilon_1$ and the 
genus, 
\begin{equation}
\label{eq:modified:hodge:compare:to:hodge}
\|\lambda\|'' \le C \|\lambda\|,
\end{equation}
i.e. the modified Hodge norm is within a multiplicative constant of
the Hodge norm.

From the definition, we have the following:
\begin{lemma}
\label{lemma:temp:modified:abs:hodge:norm}
There exists a constant $C > 1$ depending only on the genus
such that for any $t > 0$,
any $x \in \strat1$ and any $\lambda \in H^1(M,\reals)$,
\begin{displaymath}
C^{-1} e^{-2t} \|\lambda\|''_x  \le \|\lambda\|''_{a_t x} \le C e^{2t}
\|\lambda\|''_x .
\end{displaymath}
\end{lemma}

\begin{proof} From the definition of $\| \cdot \|''$,
  {}{for any $x \in \strat1$,}
\begin{equation}
\label{eq:two:prime:norm:vs:hodge:norm}
C_1^{-1} \|\lambda\|_x \le \|\lambda\|_x'' \le C_1 \ell_{hyp}(x)^{-1/2}
\|\lambda\|_x, 
\end{equation}
where $C$ depends only on the genus, and $\ell_{hyp}(x)$ is the
hyperbolic length of the shortest closed curve on $x$. It is well
known that for very short curves, the hyperbolic length is comparable
to the extremal length, see e.g. \cite{maskit}. 
It follows immediately from Kerckhoff's formula for the
Teichm\"uller distance that 
\begin{displaymath}
e^{-2t} \Ext_\gamma(x) \le \Ext_\gamma(a_t x) \le e^{2t} \Ext_\gamma(x).
\end{displaymath}
Therefore, 
\begin{equation}
\label{eq:max:change:hyperbolic:length}
C_2 e^{-2t} \ell_{hyp}(x) \le \ell_{hyp}(a_t x) \le C_2 e^{2t}
\ell_{hyp}(x), 
\end{equation}
where $C_2$ depends only on the genus. Now the lemma follows immediately from
(\ref{eq:two:prime:norm:vs:hodge:norm}), 
(\ref{eq:max:change:hyperbolic:length}) and Lemma~\ref{lemma:forni}. 
\end{proof}

One annoying feature of our definition is that for a fixed absolute
cohomology class $\lambda$, $\| \lambda \|''_x$ is not a continuous
function of $x$, as $x$ varies in a Teichm\"uller disk, due to the
dependence on the choice of short basis. To remedy this, we pick a
positive continuous $SO(2)$-bi-invariant function $\phi$ on
$SL(2,\reals)$ supported on a neighborhood of the identity $e$ such
that $\int_{SL(2,\reals)} \phi(g) \, dg = 1$, and define
\begin{displaymath}
\|\lambda\|'_x = \|\lambda\|_x + \int_{SL(2,\reals)} \|\lambda\|''_{g
  x} \, \phi(g) \, dg.  
\end{displaymath}
Then, it follows from Lemma~\ref{lemma:temp:modified:abs:hodge:norm}
that for a fixed $\lambda$, $\log \|\lambda\|'_x$ is uniformly
continuous as $x$ varies in a Teichm\"uller disk. In fact, there is a
constant $m_0$ such that for all $x \in \strat1$, all $\lambda \in
H^1(M,\reals)$ and all $t > 0$, 
\begin{equation}
\label{eq:log:abs:unif:cts}
e^{-m_0 t} \|\lambda\|'_x  \le \|\lambda\|'_{a_t x} \le e^{m_0 t}
\|\lambda\|'_x .
\end{equation}

\bold{Remark.} {}{
Even though $\|\cdot\|_x'$ is uniformly continuous as
long as $x$ varies in a Teichm\"uller disk, it may be only measurable
in general (because of the choice of short basis). 
This in the end causes our function $f_\cM$ of
Proposition~\ref{prop:main:proposition} to be discontinuous. 
}

\subsection{Relative cohomology}
\label{sec:subsec:relative}
For $c \in \rhr$ and $x =(M,\omega)\in \strat1$, let $\gp_x(c)$ 
denote the harmonic representative of
$p(c)$, where $p: \rhr \to \hr$ is the natural map. We view $\gp_x(c)$
as an element of $\rhr$. Then, (similarly to \cite{Eskin:Mirzakhani:Rafi})
we define the Hodge norm on $\rhr$ as
\begin{displaymath}
\|c\|'_x = \| p(c)\|'_x +
\sum_{(z,z') \in \Sigma\cross \Sigma} 
\left|\int_{\gamma_{z,z'}} (c - \gp_x(c))\right|,
\end{displaymath}
where $\gamma_{z,z'}$ is any path connecting the zeroes $z$ and $z'$ of
$\omega$. 
Since $c-\gp_x(c)$ represents the zero class  in absolute
cohomology, the integral does not depend on the choice of
$\gamma_{z,z'}$. Note that the $\|\cdot\|'$ norm on $\rhr$ 
is invariant under the action of $SO(2)$. 

As above, we pick a
positive continuous $SO(2)$-bi-invariant function $\phi$ on
$SL(2,\reals)$ supported on a neighborhood of the identity $e$ such
that $\int_{SL(2,\reals)} \phi(g) \, dg = 1$, and define
\begin{equation}
\label{eq:def:relative:hodge:norm}
\|\lambda\|_x = \int_{SL(2,\reals)} \|\lambda\|'_{g
  x} \, \phi(g) \, dg.  
\end{equation}
Then, the $\| \cdot \|$ norm on $\rhr$ is also invariant under the
action of $SO(2)$. 

\bold{Notational warning.} If $\lambda$ is an absolute cohomology
class, then $\|\lambda\|_x$ denotes the Hodge norm of $\lambda$ at $x$
defined in \S\ref{sec:hodge}.  If, however $\lambda$ is a relative
cohomology class, then $\|\lambda\|_x$ is defined in
(\ref{eq:def:relative:hodge:norm}). We hope the meaning will be
clear from the context.

We will use the following crude version of Lemma~\ref{lemma:forni}
(much more accurate versions are possible, especially in compact sets,
see e.g. \cite{Eskin:Mirzakhani:Rafi}). \mc{give precise reference}
\begin{lemma}
\label{lemma:growth:relative:class}
There exists a constant $m' > m_0 > 0$ 
such that for any $x \in
\cH_1(\alpha)$, any $\lambda \in \rhr$  and any $t > 0$, 
\begin{displaymath}
e^{-m' t} \| \lambda \|_x \le \|\lambda\|_{a_t x} \le e^{m' t}\|\lambda\|_{x}
\end{displaymath}
\end{lemma}

\begin{proof} 
We remark that this proof fails if we use the standard
  Hodge norm on absolute homology. 
It is enough to prove the statement assuming $0 \le t \le 1$, since
the statement for arbitrary $t$ then follows by iteration. 
It is also enough to check this for the case when $p(\lambda) = \ast
c_\gamma$, where $\gamma$ is an element of
a short basis. \mc{check}

Let $\alpha_1, \dots, \alpha_n$ be the curves with
hyperbolic length less than $\epsilon_1$.
For $1 \le k \le n$, let
$\beta_k$ be the shortest curve with $i(\alpha_k, \beta_k) =1$, where
$i(\cdot,\cdot)$ denotes the geometric intersection number. 
Let $\gamma_r$, $1 \le r \le
2g-2k$ be moderate length curves on $M$ so that the $\alpha_j$,
$\beta_j$ and $\gamma_j$ are a symplectic basis $\cS$ for
$H_1(M,\reals)$. Then $\cS$ is a short basis for $x =(M,\omega)$.

We now claim that for any curve $\gamma \in \cS$, and any $i$, $j$ 
\begin{equation}
\label{eq:key:integral}
\left| \int_{\zeta_{ij}} \ast \gamma \right| \le C \|\gamma\|''_x,
\end{equation}
where $C$ is a universal constant, and $\zeta_{ij}$ is the path
connecting the zeroes $z_i$ and $z_j$ of $\omega$ and minimizing the
hyperbolic distance. (Of course since $\ast \gamma$ is harmonic, only
the homotopy class of $\zeta_{ij}$ matters in the integral on the
left hand side of (\ref{eq:key:integral})).

It is
enough to prove (\ref{eq:key:integral}) for the $\alpha_k$ and the
$\beta_k$ (the estimate for other $\gamma \in \cS$ follows from
a compactness argument).

We can find a collar region around $\alpha_k$ as follows:
take two annuli $\{ z_k \st 1 > | z_k | > |t_k|^{1/2}\}$ and $\{ w_k \st 1 > w_k> |t_k|^{1/2} \}$ and identify the inner boundaries via the map $ w_k =
t_k/z_k$. (This coordinate system on the neighborhood
  of a boundary point in the Deligne-Mumford compactification of the
  moduli space of curves is used in e.g. \cite{Masur:WP},
  \cite[\S{3}]{wolpert}, also  \cite[Chapter 3]{fay}, \cite{Forni},
  and elsewhere. For a self-contained modern treatment see
  \cite[\S{8}]{Koch:Hubbard}).  
The hyperbolic metric $\sigma$
in the collar region is approximately $|dz|/(|z| |\log |z||)$. 
Then $\ell_{\alpha_k}(\sigma) \approx 1/|\log t_k|$,
{}{where as above, $A \approx B$ means that $A/B$ is
  bounded above and below by a constant depending only on the
  genus. (In fact, we choose the parameters $t_k$ so that
  $\ell_{\alpha_k}(\sigma) = 1/|\log t_k|$.)}
 
By \cite[Chapter 3]{fay}
any holomorphic $1$-form $\omega$ can be written in
the collar region as
\begin{displaymath}
\left( a_0(z_k+t_k/z_k,t_k) + \frac{a_1(z_k+t_k/z_k,t_k)}{z_k} \right)
\, dz_k, 
\end{displaymath}
where $a_0$ and $a_1$ are holomorphic in both variables. (We assume
here that the limit surface on the boundary of Teichm\"uller space is
fixed; this is justified by the fact that the Deligne-Mumford
compactification is indeed compact, and  if we normalize $\omega$ by fixing
its periods along $g$ disjoint curves, then in this coordinate system,
the dependence of $\omega$
on the limit surface in the boundary is continuous). 
This implies that as $t_k \to 0$, 
\begin{displaymath}
\omega = \left(\frac{a}{z_k} + h(z_k)+ O(t_k/z_k^2)\right) \, dz_k
\end{displaymath}
where $h$ is a holomorphic function which 
remains bounded as $t_k \to 0$, and the implied constant is bounded as
$t_k \to 0$. (Note that when $|z_k| \ge |t_k|^{1/2}$, $|t_k/z_k^2| \le 1$). 
Now from the condition
$\int_{\alpha_k} *c_{\beta_k} = 1$ we see that on the collar of
$\alpha_j$, 
\begin{equation}
\label{eq:c:beta:holomorphic}
c_{\beta_k} + i*\!\!c_{\beta_k} = \left(\frac{\delta_{kj}}{(2 \pi) z_j} +
h_{kj}(z_j) + O(t_j/z_j^2)\right) \, dz_j, 
\end{equation}
where the $h_{kj}$  are  holomorphic and bounded as $t_j \to 0$. (We use the notation $\delta_{kj} = 1$ if $k=j$ and zero otherwise). 
Also from the condition $\int_{\beta_k} *c_{\alpha_k} = 1$ we have
\begin{equation}
\label{eq:c:alpha:holomorphic}
c_{\alpha_k} + i*\!\!c_{\alpha_k} = \frac{i}{|\log t_j|}\left(\frac{\delta_{kj}}{z_j} +
s_{kj}(z_j) + O(t_j/z_k^2)\right) \, dz_j, 
\end{equation}
where $s_{kj}$ also remains holomorphic and is bounded as $t_j \to 0$. 
\mc{check i's and signs in the above formulas}

Then, on the collar of $\alpha_j$, 
\begin{displaymath}
\ast c_{\alpha_k} = \frac{\delta_{jk}}{|\log t_j|} d \log |z_j|^2 +
\text{ bounded $1$-form} 
\end{displaymath}
and thus, 
\begin{displaymath}
\left|\int_{\zeta_{ij}} \ast c_{\alpha_k} \right| = O(1).
\end{displaymath}
Also, on the collar of $\alpha_j$, 
\begin{displaymath}
\ast c_{\beta_k} = \frac{\delta_{jk}}{2\pi} d \arg |z_j| +
\text{ bounded $1$-form} 
\end{displaymath}
and so
\begin{displaymath}
\left|\int_{\zeta_{ij}} \ast c_{\beta_k} \right| = O(1).
\end{displaymath}
By Theorem~\ref{theorem:hodge:hyperbolic},
\begin{displaymath}
\|\ast c_{\alpha_k} \|'' \approx
O(1) \quad \text{ and } \quad \| \ast c_{\beta_k} \|'' \approx \| \ast c_{\beta_k} \|\approx \ell_{\alpha_k}(\sigma)^{1/2} \GG 1.
\end{displaymath}
Thus, (\ref{eq:key:integral}) holds for $\ast c_{\beta_k}$ and $\ast
c_{\alpha_k}$, and therefore for any $\gamma \in \cS$. By the
definition of $\| \cdot \|''$,  (\ref{eq:key:integral}) holds for any
$\lambda \in \rhr$. \mc{check}
For $0 \le t \le 1$, let $\theta_t$ denote the harmonic representative
of $p(\lambda)$ on $g_t x$. 
Then, for $0 \le t < 1$, 
\begin{align*}
\|\lambda\|'_{g_t x} & = \|p(\lambda)\|'_{g_t x} + \sum_{i,j}
\left|\int_{z_i}^{z_j} (\lambda - \theta_t) \right| && \\
& \le C \|p(\lambda)\|'_x + \sum_{i,j}\left|\int_{z_i}^{z_j} (\lambda -
  \theta_0) \right| + \sum_{i,j}\left|\int_{z_i}^{z_j} (\theta_t -
  \theta_0) \right| && \text{by (\ref{eq:log:abs:unif:cts})} \\ 
& \le C \|\lambda\|'_x + \sum_{i,j} \left|\int_{\gamma_{ij}} (\theta_t -
  \theta_0) \right| && \\ 
& \le C \|\lambda\|'_x + {}{C'} \sum_{i,j} (\| p(\lambda) \|_{g_t x}' + \|p(\lambda)
\|_x') && \text{by (\ref{eq:key:integral})}\\
& \le {}{C''} \|\lambda\|'_x && \text{by (\ref{eq:log:abs:unif:cts})}
\end{align*}
Therefore, there exists $m'$ such that for $0 \le t \le 1$ and any
$\lambda \in \rhr$, 
\begin{displaymath}
\|\lambda\|_{g_t x} \le e^{m' t} \|\lambda\|_x.
\end{displaymath}
This implies the lemma for all $t$. 
\end{proof}

In the sequel we will need to have a control of the matrix
coefficients of the cocycle. Let $x\in\hcal_1(\alpha)$ and $t\in
\reals$ we let $A(x,t) \equiv A(x,a_t)$ denote the cocycle. Using the
map $p$ above we may write \mc{fix this notation}
\begin{equation}
\label{eq:cocyle}
A(x,t)=\left(\begin{array}{cc} I & U(x,t)\\ 0 &
    S(x,t)\end{array}\right)
\end{equation}
(Note that {}{since we are labelling the zeroes of $\omega$,}
the action of the cocycle on $\ker p$ is trivial.)

The following is an immediate corollary of
Lemma~\ref{lemma:growth:relative:class}: 
\begin{lemma}
\label{lemma:unip-cont}
There is some $m'\in\bbn$ such that for all $x\in\hcal_1(\alpha)$ and
all $t \in \reals$ we have 
\begin{displaymath}\|U(x,t)\|\leq e^{m'|t|},\end{displaymath}
where 
\begin{equation}
\label{eq:def:Uxt}
\|U(x,t)\| \equiv \sup_{c \in \rhr} \frac{\|\gp_x(c) - \gp_{a_t
    x}(c)\|}{\|p(c)\|'_x}.  
\end{equation}
Note that since $\gp_x(c) - \gp_{a_t x}(c) \in \ker p$, $\|\gp_x(c) -
\gp_{a_t x}(c)\|_y$ is independent of $y$. 
\end{lemma}

Suppose $L \subset \rhr$ is a subspace such that $p(L) \subset
H^1(M,\reals)$ is symplectic (in the sense that the intersection form
restricted to $p(L)$ is non-degenerate). 
Let $p(L)^\perp$ denote the symplectic complement of $p(L)$ in
$H^1(M,\reals)$. Suppose $x \in \cH_1(\alpha)$. For any $c \in \rhr$
we may write 
\begin{displaymath}
c = h + c' +v,
\end{displaymath}
where $h$ is harmonic with $p(h) \in p(L)^\perp$, $v \in L$ and
$c' \in \ker p$. This decomposition is not unique since for $u \in L
\cap \ker p$, we can replace $c'$ by $c' + u$ and $v$ by
$v-u$. We denote the $c'$ with smallest possible $\|\cdot \|_x$ 
norm by $c'_L$. Thus, we
have the decomposition
\begin{equation}
\label{eq:c:decomp}
c = \gp_{x,L}(c) + c'_L + v,
\end{equation}
where $\gp_{x,L}(c)$ is the harmonic representative at $x$ of
$p_L(c) \equiv \pi_{L^\perp}(p(c))$, $c'_L \in \ker p$, $v \in L$, and $c'_L$
has minimal norm.

Define $\nu_{x,L}:\rhr\rightarrow\bbr$ by
\begin{displaymath}
\nu_{x,L}(c)=
\begin{cases}
\max\{\|c'_L\|_x,(\|p_{L}(c)\|_{x}')^{1/2}\} 
    & \text{if}\h\h \max\{\|c'_L\|_{x},\|p_L(c)\|'_{x}\}\leq 1 \\
    \hspace{0.7in} 1 & \text{otherwise.}
\end{cases}
\end{displaymath}
We record (without proof) some simple properties of $\nu_{x,L}$. 
\begin{lemma}
\label{lemma:properties:nu:x:cL}
We have 
\vspace{-0.1in}
\begin{itemize}
\item[{\rm (a)}] $\nu_{x,L}(c) = 0$ if and only if $c \in L$. 
\item[{\rm (b)}] For $v \in L$, $\nu_{x,L}(c+v) =
  \nu_{x,L}(v)$. 
\item[{\rm (c)}] For $v' \in \ker p$, $\nu_{x,L}(c) - \|v'\|_x \le
  \nu_{x,L}(c + v') \le \nu_{x,L}(c) + \|v'\|_x$. 
\end{itemize}
\end{lemma}
In view of Lemma~\ref{lemma:properties:nu:x:cL}, for an affine
subspace $\cL = v_0 + L$ of $\rhr$, 
we can define $\nu_{x,\cL}(c)$ to be $\nu_{x,L}(c-v_0)$.

Extend $\nu_{x,\cL}$ to $H^1(M,\Sigma,\cx)$ by  
\begin{displaymath}
\nu_{x,\cL}(c_1+ i c_2)=\max\{\nu_{x,\cL}(c_1),\nu_{x,\cL}(c_2)\}.
\end{displaymath}

For an affine subspace $\cL \subset \rhr$, let $\cL_\cx \subset
H^1(M,\Sigma,\cx)$ denote the complexification $\cx \tensor \cL$. 
We use the notation (here we are working in period coordinates)
\begin{displaymath}
d'(x,\cL) = \nu_{x,\cL}(x - v)
\end{displaymath}
where $v$ is any vector in $\cL_\cx$ (and the choice of $v$ does not
matter by Lemma~\ref{lemma:properties:nu:x:cL} (b)). Note that
$d'(\cdot, \cL)$ is defined only if $\cL = v_0 + L$ where $p(L)$ is
symplectic. 
We think of
$d'(x,\cL)$ as measuring the distance between $x$ and $\cL_\cx \subset
H^1(M,\cx)$. In view of Lemma~\ref{lemma:growth:relative:class}, 
we have for all $t > 0$
\begin{equation}
\label{eq:growth:dprime}
e^{-m' t} d'(x,\cL) \le d'(a_t x, a_t \cL) \le e^{m' t} d'(x,\cL).
\end{equation}

Recall that $\delta > 0$ is defined in the beginning of  
\S\ref{sec:subsec:absolute}.
\begin{lemma}
\label{lemma:ineq-rel}
Let the notation be as above. Then, there exists constants
$C_0 > 0$,
$L_0 > 0$, $\eta_0' > 0$, $\eta_3 > 0$, $t_0' > 0$
and continuous functions $\kappa_1: \cH_1(\alpha) \cross \reals^+ \to
\reals^+$ and $b: \reals^+ \to \reals^+$ such that 
\begin{itemize}
\item $\kappa_1(x,t) \le
C_0 e^{m' \delta t}$ for all $x \in \cH_1(\alpha)$ and all $t > 0$
\item $\kappa_1(x,t) \le e^{-\eta_3 t}$ for all $x
    \in \strat1$ and
  $t > t_0'$ with $\log u(x) < L_0 + \eta_0' t$, 
\end{itemize}
so that for any affine
subspace  $\cL \subset H^1(M,\Sigma, \reals)$ 
such that the projection of the linear part of $\cL$ to $\hr$ is symplectic
we have
\begin{equation}
\label{eq:ineq-rep1}
\frac{1}{2\pi}\int_0^{2\pi}\frac{d\theta}{d'(a_t r_\theta x, a_t
  r_\theta
  \cL)^\delta}\leq\frac{\kappa_1(x,t)}{d'(x,\cL)^\delta}+ b(t)
\end{equation}
\end{lemma}

\begin{proof} Suppose $d'(x,\cL) \ge 1$, or $d'(a_t r_\theta x, a_t
  r_\theta \cL) \ge 1$ for some $\theta \in [0,2\pi]$. Then,
  (\ref{eq:ineq-rep1}) with $b(t) = e^{2m'\delta t}$ follows
  immediately from (\ref{eq:growth:dprime}). Therefore, we may assume
  that
\begin{equation}
\label{eq:initial:assumption}
d'(x,\cL) < 1, \quad\text{ and } \quad d'(a_t r_\theta x, a_t r_\theta \cL) < 1
\text{ for all $\theta$.}
\end{equation}
Therefore, in particular,
\begin{equation}
\label{eq:def:v:vprime}
d'(x,\cL) = \nu_{x,\cL}(v) = \max(\|v'\|_x, (\|p(v)\|'_x)^{1/2}), 
\end{equation}
where
\begin{equation}
\label{eq:basic:decomp:v}
v = \gp_x(v) + v', \quad \text{ $p(v) \in p(L_\cx)^\perp$, $L$ is
  the linear part of $\cL$, 
  and $v' \in \ker p$. }
\end{equation}
We remark that the main difficulty of the proof of this lemma 
is to control the
interaction between absolute and pure relative cohomology. The strategy
is roughly as follows: we quickly reduce to the case where $v$ is extremely
small. Then, 
if the size of the absolute part $\|p(v)\|'_x$ is
comparable to the size of pure relative part 
$\|v'\|$, then the quantities $d(x,\cL)$ and $d(a_t
r_\theta x, a_t r_\theta \cL)$ are all controlled by the absolute part
(because of the square root in (\ref{eq:def:v:vprime})). In fact, the
only situation in which the pure relative part $v'$ has an effect is
when $\|p(v)\|'_x$ is essentially smaller then $\|v'\|^2$ (so it is
tiny). In this regime, the influence of the absolute part on the
relative part is very small, in view of
Lemma~\ref{lemma:unip-cont}. This allows us to separate the
contribution of absolute and pure relative cohomology in all
cases: for a precise statement, see (\ref{eq:separate:abs:rel})
below. We now give the detailed implementation of this strategy. 

Suppose $d'(x,\cL) = \nu_{x,\cL}(v) 
\ge \frac{1}{2} e^{-3 m' t}$. Then, using
  (\ref{eq:growth:dprime}) we have the crude estimate 
\begin{displaymath}
d'(a_t r_\theta x, a_t r_\theta \cL)^{-\delta} \le 
d'(a_t r_\theta x, a_t r_\theta \cL)^{-1} \le 2 e^{5 m' t} 
\end{displaymath}
and thus (\ref{eq:ineq-rep1}) holds with $b(t) =  2 e^{5 m' t}$. Hence, we may
assume that $\nu_{x,\cL}(v) < \frac{1}{2} e^{-3 m' t}$. Then, 
\begin{align}
\label{eq:tmp:forlater}
e^{2m't} (\|p(a_t r_\theta v)\|_{a_t r_\theta
  x}')^{1/2} & \le  
 e^{(2m'+0.5)t} (\|p(v)\|_{a_t r_\theta
   x}')^{1/2} && \text{by    (\ref{eq:action:SL2})} \notag \\
 & \le e^{(2m'+0.5 +0.5m_0)t} (\|p(v)\|'_x)^{1/2} &&\text{by
   (\ref{eq:log:abs:unif:cts})}\notag \\ 
 & \le e^{3 m' t} \nu_{x,\cL}(v)  && \text{since $m' >
     m_0 > 1$} \notag \\
 & \le \tfrac{1}{2}. && 
\end{align}
Let us introduce the notation, for $u \in \ker p$,
\begin{displaymath}
\|u\|_{\cL} = \inf \{ \|u-w\| \st w \in \cL \cap \ker p \}. 
\end{displaymath}
Then, by (\ref{eq:initial:assumption}), \mc{explain}
\begin{equation}
\label{eq:dprime:evolution}
d'(a_t r_\theta x, a_t r_\theta \cL) = \max( (\|p(a_t r_\theta
v)\|_{a_t r_\theta x}')^{1/2}, \|a_t r_\theta v - \gp_{a_t r_\theta x}(a_t r_\theta v)\|_{a_t r_\theta \cL}).
\end{equation}
But, 
\begin{align}
\label{eq:tmp:reverse:triangle}
 \|a_t r_\theta v - & \gp_{a_t r_\theta x} (a_t r_\theta v)\|_{a_t r_\theta
  \cL} = && \notag \\
& = \| a_t r_\theta (v' + \gp_x(v)) - \gp_{a_t r_\theta x}(a_t
r_\theta v) \|_{a_t r_\theta \cL} && \text{by
  (\ref{eq:basic:decomp:v})} \notag \\
& = \| a_t r_\theta v' + \gp_x(a_t r_\theta v) -\gp_{a_t r_\theta
  x}(a_t r_\theta v) \|_{a_t r_\theta \cL} && \notag \\
& \ge \| a_t r_\theta v' \|_{a_t r_\theta \cL} - 
\| \gp_x(a_t r_\theta v) -\gp_{a_t r_\theta x}(a_t r_\theta v) \| &&
\text{by the reverse triangle inequality}\notag \\
& \ge \| a_t r_\theta v' \|_{a_t r_\theta \cL} - \|U(r_\theta x, t)\| 
\|p( a_t r_\theta v)\|'_x && \text{by
  (\ref{eq:def:Uxt})} \notag \\ 
& \ge \| a_t r_\theta v' \|_{a_t r_\theta \cL} - e^{2m' t} 
\|p( a_t r_\theta v)\|'_{a_t r_\theta x} && \text{by
  Lemma~\ref{lemma:unip-cont}.} 
\end{align}
Therefore,
\begin{align}
\label{eq:separate:abs:rel}
& d'(a_t r_\theta x, a_t r_\theta \cL) = \notag \\
& = \max\left( (\|p(a_t r_\theta
v)\|'_{a_t r_\theta x})^{1/2}, \|a_t r_\theta v -
  \gp_{a_t r_\theta 
  x}(a_t r_\theta v)\|_{a_t r_\theta \cL}\right) && \text{by
(\ref{eq:dprime:evolution})} \notag \\
& \ge \frac{1}{2} \left( \|a_t r_\theta v - \gp_{a_t r_\theta
  x}(a_t r_\theta v)\|_{a_t r_\theta \cL} + (\|p(a_t r_\theta
v)\|'_{a_t r_\theta x})^{1/2}\right) \notag \\
& \ge \frac{1}{2}\left(\| a_t r_\theta v' \|_{a_t r_\theta \cL} - e^{2m' t} 
\|p( a_t r_\theta v)\|'_{a_t r_\theta x} + (\|p(a_t r_\theta
v)\|'_{a_t r_\theta x})^{1/2}\right) && \text{by
(\ref{eq:tmp:reverse:triangle})}\notag \\
& = \frac{1}{2}\left(\|a_t r_\theta v'\|_{a_t r_\theta \cL} + (\|p(a_t r_\theta
v)\|'_{a_t r_\theta x})^{1/2}(1-e^{2 m't} (\|p(a_t
r_\theta v)\|'_{a_t   r_\theta x})^{1/2})\right) \notag \\
& \ge \frac{1}{2}\left(\|a_t r_\theta v'\|_{a_t r_\theta \cL} + \frac{1}{2}
(\|p(a_t r_\theta
v)\|'_{a_t r_\theta x})^{1/2}\right) &&\text{by
(\ref{eq:tmp:forlater})} \notag \\
&  \ge 
\frac{1}{2}\left(\|a_t r_\theta v'\|_{a_t r_\theta \cL} + \frac{1}{2}
(\|p(a_t r_\theta
v)\|_{a_t r_\theta x})^{1/2}\right)
&& \text{since $\| \cdot \|' \ge \|
\cdot \|$}
\end{align}
However, since the action of the cocycle on $\ker p$ is trivial, $v'
\in \ker p$ and $\cL$ is invariant, 
\begin{displaymath}
\|a_t r_\theta v'\|_{a_t r_\theta \cL} = \|a_t r_\theta v'\|.
\end{displaymath}
Then, {}{(with $v$ and $v'$ as in (\ref{eq:def:v:vprime}) and
(\ref{eq:basic:decomp:v})),} 
\begin{align*}
\frac{1}{2\pi} \int_{0}^{2\pi} &\frac{d\theta}{d'(a_t r_\theta x,
  a_t r_\theta \cL)^\delta} \le \\
& \le \frac{1}{2\pi} \int_0^{2\pi} 4
\min\left(  
  \frac{1}{\|a_t r_\theta v')\|^{\delta}}, \frac{1}{\|a_t
    r_\theta p(v) 
  \|^{\delta/2}} \right) \, d\theta \\
& \le 4 \min \left( \frac{1}{2\pi} \int_0^{2\pi} 
  \frac{d\theta}{\|a_t r_\theta v'\|^{\delta}}, \frac{1}{2\pi}
  \int_0^{2\pi} \frac{d\theta}{\|a_t
    r_\theta p(v) 
  \|^{\delta/2}} \right) \, d\theta \\
& \le 4 \min\left(
  \frac{e^{-k_2(\delta)t}}{\|v'\|^\delta_x},    \frac{\min(C_0, \psi_x(p(v))^{\delta/2}\kappa(x,t))}{\|p(v)\|_{x}^{\delta/2}} \right) && 
\text{\hspace{-0.7in} by Lemma~\ref{lemma:EMM51} 
 and Lemma~\ref{lemma:ineq:absolute:thick:part}} \\
\end{align*}

Let $\eta_0' > 0$ be a constant to be chosen later. Suppose $\log u(x)
< L_0 + \eta_0' t$. 
By Theorem~\ref{theorem:fast:return} there 
exists $\theta \in [0,2\pi]$ and $\tau \le m'' \log u(x)$ such that
$x' \equiv a_\tau r_\theta x \in K_\rho$. Then, 
\begin{displaymath}
\tau \le m'' L_0 + m'' \eta_0' t.
\end{displaymath}
Then, for any $v$, 
\begin{displaymath}
\|p(v)\|'_{x} \le e^{m_0 \tau} \|p(v)\|'_{x'} \le C_0 e^{m_0 \tau}
\|p(v)\|_{x'} \le C_0 e^{(m_0 + 2) \tau} \|p(v)\|_{x}
\end{displaymath}
Therefore, by Lemma~\ref{lemma:ineq:absolute:thick:part} (b), 
\begin{displaymath}
\frac{\kappa(x,t)}{\|p(v)\|_x^{\delta/2}} \le e^{-\eta t} C_0
e^{(\delta/2)(m_0+2) (m'' L_0 + m'' \eta_0' t)}
(\|p(v)\|_x')^{-\delta/2} \le 
e^{-(\eta/2) t} (\|p(v)\|_x')^{-\delta/2}, 
\end{displaymath}
provided $(\delta/2) m'' \eta_0' < \eta/2$ and $t_0'$ is sufficiently
large. 

Let $v$ be as defined in (\ref{eq:def:v:vprime}). Note that $x+v \in
\cL_\cx$ (in period coordinates), and $p(v)$  
is (symplectically) orthogonal to $p(\cL_\cx)$. Let $w = a_\tau
r_\theta v$. Then, since $\cL$ is invariant, $p(w)$ is symplectically
orthogonal to $p(\cL_\cx)$.  Therefore,
$\psi_{x'+w}(p(w)) = 1$. Also, by   definition, the subspace $E(x')$
varies continuously with $x'$, hence for any $y \in \cL_\cx$,
\begin{displaymath}
\lim_{x' \to y} \psi_{x'}(p(w)) = 1. 
\end{displaymath}
Since we are assuming that $d'(x',\cL)$ is small (in fact
$d'(x,\cL) \le \tfrac{1}{2} e^{-m't}$ and $\tau \ll t$), 
we conclude that 
$\psi_{x'}(p(w))$ is uniformly bounded. Therefore, 
\begin{displaymath}
\psi_x(p(v))^{\delta/2} \le e^{C \eta_0' (\delta/2) 2\tau} \le e^{(\eta/4)t}
\end{displaymath}
provided $\eta_0'$ is small enough. \mc{give more detail} 
Thus, we get, for $t > t_0'$ and $x \in \strat1$ so that $\log u(x) <
L_0 + \eta_0' t$, 
\begin{displaymath}
\frac{1}{2\pi} \int_{0}^{2\pi} \frac{d\theta}{d'(a_t r_\theta x,
  a_t r_\theta \cL)^\delta} \le
4 \min\left(\frac{
  e^{-k_2(\delta)t}}{\|v'\|^\delta_x},
\frac{e^{-(\eta/4)t}}{(\|p(v)\|'_x)^{\delta/2}} \right) 
\end{displaymath}
The estimate (\ref{eq:ineq-rep1}) now follows. \mc{give more detail}
\end{proof}


\section{The sets $J_{k,\cM}$}
\label{sec:regions}

{}{Let $\tilde{\cH}_1(\alpha)$ denote the space of
markings of translation surfaces in $\cH_1(\alpha)$ with the zeroes
labelled. Then $\tilde{\cH}_1(\alpha)$ is a bundle over (a finite
cover of) the Teichm\"uller space of Riemann surfaces, or
alternatively a stratum of the Teichm\"uller space of holomoporphic
$1$-forms.
} 

Fix $0<\rho < 1/2$ so that Theorem~\ref{theorem:all:measures:return} and
Theorem~\ref{theorem:fast:return} hold.  Let $K_\rho$ be as in
Theorem~\ref{theorem:all:measures:return}  and let $K' := \{ x \st
{\rm d}(x, K_{0.01}) \le 1 \}$ where ${\rm d}$ denotes the Teichm\"uller distance.  Then, $K'$ is a
compact subset of $\strat1$. We lift $K'$ to a compact subset of
{}{the Teichm\"uller space $\tilde{\cH}_1(\alpha)$,}
which we also denote by $K'$.

\begin{definition}[Complexity]
\label{def:number:sheets}
For an affine invariant submanifold $\cM \subset \strat1$, let
$n(\cM)$ denote the smallest integer such that $\cM \cap K'$ is
contained in a union of at most $n(\cM)$ affine subspaces. We call
$n(\cM)$ the ``complexity'' of $\cM$.
\end{definition}
Since $\cM$ is closed and $K'$ is compact, $n(\cM)$ is always
finite. Clearly $n(\cM)$ depends also on the choice of $K'$, but since
$K'$ is fixed once and for all, we drop this dependence from the
notation. 

\begin{lemma}
\label{lemma:def:J:Jprime}
Let $\cM$ be an affine manifold, and let $\tilde{\cM}$ be a lift of
$\cM$ to {}{the Teichm\"uller space $\tilde{\cH}_1(\alpha)$.} 
For $x \in \cH_1(\alpha)$, let 
\begin{displaymath}
J_{k,\cM}(x) = \{ \cL \st d'(\cL,x) \le u(x)^{-k}, \quad \text{ $\cL$ is
  an affine subspace tangent to $\tilde{\cM}$ }\}. 
\end{displaymath}
Then, there exists $k > 0$, 
  depending only on  $\alpha$ such that for any affine manifold $\cM \subset
  \cH_1(\alpha)$, 
\begin{displaymath}
|J_{k,\cM}(x)| \le n(\cM)
\end{displaymath}
where $|J_{k,\cM}(x)|$ denotes the cardinality of $J_{k,\cM}(x)$, and
$n(\cM)$ is as in Definition~\ref{def:number:sheets}. 
\end{lemma}

\mc{do we need to say ``embedded'' in the statement here?}

\begin{proof} 
We lift $x$ to {}{the Teichm\"uller space
  $\tilde{\cH}_1(\alpha)$.} Working in period
  coordinates, let
\begin{displaymath}
B'(x,r) = \{ x + h + v \st \text{$h$ harmonic, $v \in \ker p$, 
$\max(\| h\|_x^{1/2}, \|v\|_x) \le r$} \} 
\end{displaymath}
For every $x \in \strat1$, there exists $r(x)> 0$ such
that $B'(x,r(x))$ is embedded (in the sense that the projection from
{}{the 
Teichm\"uller space $\tilde{\cH}_1(\alpha)$ to the Moduli space
$\cH_1(\alpha)$,} 
restricted to $B'(x,r(x))$ is injective). 
Furthermore, we may choose 
$r(x)> 0$ small enough 
so that the periods on $B'(x,r(x))$ are a coordinate system
(both on {}{the Teichm\"uller space $\tilde{\cH}_1(\alpha)$} and on the Moduli space $\cH_1(\alpha)$).
Let $r_0 = \inf_{x \in K_\rho} r(x)$. By compactness of $K_\rho$, $r_0 >
0$. Then, choose $k_0$ so that 
\begin{equation}
\label{eq:choice:k}
2^{m'' m' -k_0} < r_0. 
\end{equation}
where $m''$ be as in
Theorem~\ref{theorem:fast:return}, and $m'$ is as in
(\ref{eq:growth:dprime}). 

We now claim that for any $k > k_0$ and any $x \in \strat1$, 
$B'(x,u(x)^{-k_0})$ is embedded. 
Suppose not, then there exist $x \in \strat1$ and 
$x_1, x_2 \in B'(x,u(x)^{-k_0})$ such
that $x_2 = \gamma x_1$ for some $\gamma$ in the mapping class
group. Write
\begin{displaymath}
x_i = x + h_i + v_i, \quad\text{$h_i$ harmonic, $v_i \in \ker p$, 
$\max(\| h_i\|_x^{1/2}, \|v_i\|_x) \le u(x)^{-k_0}$} 
\end{displaymath}
By Theorem~\ref{theorem:fast:return} there
exists $\theta \in [0,2\pi]$ and $\tau \le m'' \log u(x)$ such that
$x' \equiv a_\tau r_\theta x \in K_\rho$.

Let $x_i' = a_\tau t_\theta x_i$. Then, by
Lemma~\ref{lemma:growth:relative:class}  we have
\begin{displaymath}
\max(\| h_i\|_{x'}^{1/2}, \|v_i\|_{x'}) \le e^{-m'\tau}u(x)^{-k_0}
\le u(x)^{m' m'' -k_0} \le 2^{m m' -k_0} \le r_0
\end{displaymath}
where for the last estimate we used (\ref{eq:choice:k}) and the fact
that $u(x) \ge 2$. 
Thus, both $x_1'$ and $x_2'$ belong to $B'(x',r_0)$, which is embedded
by construction, contradicting the fact that $x_2' = \gamma x_1'$. 
Thus, $B'(x,u(x)^{-k})$ is embedded. 

Now suppose $\cL \in J_{k,\cM}(x)$, so that 
\begin{displaymath}
d'(x, \cL) \le u(x)^{-k}.
\end{displaymath}
Write $\cL' = a_\tau r_\theta \cL$. Then, by (\ref{eq:growth:dprime}),
\begin{displaymath}
d'(x',\cL') \le e^{m' \tau} u(x)^{-k} \le u(x)^{m'' m'} u(x)^{-k}
< r_0, 
\end{displaymath}
Hence, $\cL'$ intersects $B'(x',r_0)$. Furthermore, since $B'(x',r_0)$
and $B'(x,u(x)^{-k})$ are embedded, there is a one-to-one map
between subspaces contained in $J_{k,\cM}(x)$ and subspaces
intersecting $B'(x',r_0)$. 

Since $x' \in K_\rho$, and $r_0
< 1$, $B'(x',r_0) \subset K'$. 
Hence, there are at most $n(\cM)$ possibilities for $\cL'$, and hence
at most $n(\cM)$ possibilities for $\cL$. 
\end{proof}


\section{Standard Recurrence Lemmas}

\begin{lemma}
\label{lemma:abstract:recurrence:properties}
For every $\sigma > 1$ there exists a constant $c_0
= c_0(\sigma) > 0$ such that the following holds:
Suppose $X$ is a space on which $SL(2,\reals)$ acts, and suppose $f: X
\to [2,\infty]$ is an $SO(2)$-invariant 
function with the following properties:
\begin{itemize}
\item[{\rm (a)}] For all $0 \le t \le 1$ and all $x \in X$, 
\begin{equation}
\label{eq:abstract:condition:a}
\sigma^{-1} f(x) \le f(a_t x) \le \sigma f(x). 
\end{equation}
\item[{\rm (b)}] There exists $\tau > 0$ and
  $b_0 > 0$ such that for all $x \in X$, 
\begin{displaymath}
A_\tau f(x) \le c_0 f(x) + b_0.
\end{displaymath}
\end{itemize}
Then, 
\begin{itemize}
\item[{\rm (i)}] For all $c < 1$ there exists $t_0 > 0$ (depending on
  $\sigma$, and $c$) and $b > 0$ (depending only on $b_0$, $c_0$ and $\sigma$) 
  such that for all $t > t_0$ and all $x \in X$, 
\begin{displaymath}
(A_t f)(x) \le c f(x) + b.
\end{displaymath}
\item[{\rm (ii)}] There exists $B > 0$ (depending only on $c_0$, $b_0$ and
  $\sigma$) such that for all $x \in X$, there exists $T_0 =
  T_0(x,c_0,b_0,\sigma)$ such that for all $t > T_0$, 
\begin{displaymath}
(A_t f)(x) \le B.
\end{displaymath}
\end{itemize}
\end{lemma}
For completeness, we include the proof of this lemma. It is
essentially taken from
\cite[\S{5.3}]{EMM1}, specialized to the case $G = SL(2,\reals)$. The
basic observation is the following standard fact from hyperbolic
geometry: 
\begin{lemma}
\label{lemma:hyperbolic}
There exist absolute constants $0 < \delta' < 1$ and
$\delta > 0$ such that for any $t > 0$, any $s > 0$ and any $z \in
\half$, for at least $\delta'$-fraction of $\phi \in [0,2\pi]$, 
\begin{equation}
\label{eq:hyperbolic:triangle}
t + s -\delta \le d(a_t r_\phi a_s z,z) \le t+s,
\end{equation}
where $d(\cdot, \cdot)$ is the hyperbolic distance in $\half$,
normalized so that $d(a_t r_\theta z, z) = t$. 
\end{lemma}

\begin{corollary}
\label{cor:AtAs:inequality}
Suppose $f: X \to [1,\infty]$ satisfies
(\ref{eq:abstract:condition:a}). Then, there exists $\sigma'>1$
depending only on $\sigma$ such that for any $t > 0$, $s > 0$ and any
$x \in X$, 
\begin{equation}
\label{eq:AtAs:inequality}
(A_{t+s} f)(x) \le \sigma' (A_t A_s f)(x).
\end{equation}
\end{corollary}

\begin{proof}[Outline of proof]
Fix $x \in \strat1$. For $g \in SL(2,\reals)$, let $f_x(g) = f(gx)$,
and let 
\begin{displaymath}
\tilde{f}_x(g) = \int_0^{2\pi} f( g r_\theta x) \, d\theta. 
\end{displaymath}
Then, $\tilde{f}_x: \half \to [2,\infty]$ is a spherically symmetric
function, i.e. $\tilde{f}_x(g)$ depends only on $d(g \cdot o,o)$ where $o$ is
{}{the point fixed by $SO(2)$.}

We have
\begin{equation}
\label{eq:AtAs:f}
(A_t A_s f)(x) = \frac{1}{2\pi} \int_0^{2\pi} \frac{1}{2\pi}
\int_{0}^{2\pi} f(a_t r_\phi a_s r_\theta x) \, d\phi \, d\theta =
\frac{1}{2\pi} \int_0^{2\pi} \tilde{f}_x(a_t r_\phi a_s).
\end{equation}
By Lemma~\ref{lemma:hyperbolic}, for at least $\delta'$-fraction of
$\phi \in [0,2\pi]$, (\ref{eq:hyperbolic:triangle}) holds. 
Then, by (\ref{eq:abstract:condition:a}), for at least
$\delta'$-fraction of $\phi \in [0,2\pi]$, 
\begin{displaymath}
\tilde{f}_x(a_t r_\phi a_s ) \ge \sigma_1^{-1} \tilde{f}_x(a_{t+s}) 
\end{displaymath}
where $\sigma_1 = \sigma_1(\sigma, \delta) > 1$.  
Plugging in to
(\ref{eq:AtAs:f}), we get
\begin{displaymath}
(A_t A_s f)(x) \ge (\delta' \sigma_1^{-1}) 
\tilde{f}_x(a_{t+s}) = (\delta' \sigma_1^{-1}) (A_{t+s} f)(x),
\end{displaymath}
as required. 
\end{proof}

\begin{proof}[Proof of Lemma~\ref{lemma:abstract:recurrence:properties}]
Let $c_0(\sigma)$ be such that $\kappa \equiv c_0 \sigma' < 1$, where
$\sigma'$ is as in Corollary~\ref{cor:AtAs:inequality}. Then, 
for any $s \in \reals$ and for all $x$, 
\begin{align*}
(A_{s+ \tau} f)(x) & \le \sigma' A_s (A_\tau f)(x) && \text{by
  (\ref{eq:AtAs:inequality})} \\
& \le \sigma' A_s (c_0 f(x) + b_0) && \text{by condition (b)} \\
&  = \kappa (A_s f)(x) + \sigma' b_0 && \text{since $\sigma' c_0 = \kappa$.}
\end{align*}
Iterating this we get, for $n \in \natls$ 
\begin{displaymath}
(A_{n \tau} f)(x) \le \kappa^n f(x) + \sigma' b_0 + \kappa \sigma'
b_0 + \dots + \kappa^{n-1} \sigma' b_0 
\le \kappa^n f(x) + B,
\end{displaymath}
where $B = \frac{\sigma' b_0}{1-\kappa}$.
Since $\kappa < 1$, $\kappa^n f(x) \to 0$ as $n \to \infty$. Therefore
both (i) and (ii) follow for $t \in \tau \natls$. The general case of
both (i) and (ii) then follows by applying again condition (a). 
\end{proof}


\section{Construction of the function}
\label{sec:function}

Note that by Jensen's inequality, for $0 < \epsilon < 1$, 
\begin{equation}
\label{eq:At:Jensen}
A_t (f^\epsilon) \le (A_t f)^\epsilon
\end{equation}
Also, we will repeatedly use the inequality
\begin{equation}
\label{eq:a:plus:b:epsilon}
(a +b)^\epsilon  \le a^\epsilon + b^\epsilon
\end{equation}
valid for $\epsilon < 1$, $a \ge 0$, $b \ge 0$.

Fix an affine invariant submanifold $\cM$, and let $k$ be as in
Lemma~\ref{lemma:def:J:Jprime}. For $\epsilon > 0$, let
\begin{displaymath}
s_{\cM,\epsilon}(x) = \begin{cases}
\sum\limits_{\cL \in J_{k,\cM}(x)} d'(x,\cL)^{-\epsilon\delta}, & \text{ if
  $J_{k,\cM}(x) \ne   \emptyset$} \\
 0 & \text{ otherwise.}
\end{cases}
\end{displaymath}
where $\delta > 0$ is as in Lemma~\ref{lemma:ineq-rel}.  

\begin{prop}
\label{prop:fM:inequality}
Suppose $\cM \subset \cH_1(\alpha)$ 
is an affine manifold and $0 < c < 1$. 
For $\epsilon > 0$ and $\lambda >
0$, let
\begin{displaymath}
f_{\cM}(x) = s_{\cM,\epsilon}(x) u(x)^{1/2} + \lambda u(x). 
\end{displaymath}
Then, $f_\cM$ is $SO(2)$-invariant, and $f(x) = +\infty$ if and only
if $x \in \cM$. 
Also, if $\epsilon$ is sufficiently small (depending on $\alpha$) and 
$\lambda$ is sufficiently large (depending on $\alpha$, $c$ and $n(\cM)$), 
there exists $t_1 > 0$ (depending on $n(\cM)$ and $c$) 
such that for all $t\geq t_1$ we have
\begin{equation}
\label{eq:fM:inequality}
A_t f_\cM(x)< c f_\cM(x)+b,
\end{equation}
where $b = b(\alpha, n(\cM))$. 
\end{prop}

The proof of Proposition~\ref{prop:fM:inequality} will use
Lemma~\ref{lemma:abstract:recurrence:properties}. 
Thus, in order to prove Proposition~\ref{prop:fM:inequality}, it is
enough to show that $f_{\cM}$ satisfies conditions (a) and (b) of
Lemma~\ref{lemma:abstract:recurrence:properties}. We start with the
following:
\begin{claim}
\label{claim:fM:condition:a:works}
For $\epsilon > 0$ sufficiently small, and $\lambda > 0$ sufficiently
large, $f_\cM$ satisfies condition (a) of
Lemma~\ref{lemma:abstract:recurrence:properties}, with $\sigma =
\sigma(k, m, m')$. 
\end{claim}
\begin{proof}[Proof of Claim~\ref{claim:fM:condition:a:works}] 
We will choose $\epsilon < 1/(2 k \delta)$. 
Suppose $x \in \strat1$ and $0 \le t < 1$. We consider three sets of
subspaces: 
\begin{displaymath}
\Delta_1 = \{ \cL \in J_{k,\cM}(x) \st a_t \cL \in J_{k,\cM}(a_t x) \},
\end{displaymath}
\begin{displaymath}
\Delta_2 = \{ \cL \in J_{k,\cM}(x) \st a_t \cL  \not\in J_{k,\cM}(a_t x) \},
\end{displaymath}
\begin{displaymath}
\Delta_3 = \{ \cL \not\in J_{k,\cM}(x) \st a_t \cL \in J_{k,\cM}(a_t
x) \}. 
\end{displaymath}
We remark that the rest of the proof is a routine verification. Note
that the cardinality of all $\Delta_i$ is bounded by $n(\cM)$ which is
fixed. For any
$0 \le t \le 1$, in view of (\ref{eq:growth:dprime}), the 
contribution of each $\cL$ in $\Delta_1$ at $a_t x$ is within a fixed
multimplicative factor of the contribution at $x$. Furthermore, if
$\cL \in \Delta_2 \cup \Delta_3$, then in view of
(\ref{eq:growth:dprime}), $d'(x,\cL)$ is bounded from below by a
negative power of $u(x)$, and then (with the proper choice of
parameters), it's contribution to both $f_\cM(x)$ and $f_\cM(a_t x)$
is negligible. We now give the details. 

Let 
\begin{displaymath}
S_i(x) = \sum_{\cL \in \Delta_i} d'(x,\cL)^{-\epsilon \delta}. 
\end{displaymath}
Then, 
\begin{displaymath}
s_{\cM,\epsilon}(x) = S_1(x) + S_2(x) \qquad s_{\cM,\epsilon}(a_t x) =
S_1(a_t x) + S_3(a_t x). 
\end{displaymath}
For $\cL \in \Delta_1$, by (\ref{eq:growth:dprime})
with $0 \le t \le 1$,
\begin{displaymath}
e^{-m' \epsilon \delta} d'(x,\cL)^{-\epsilon \delta} \le d'(a_t x, a_t
\cL)^{-\epsilon\delta} \le e^{m'\epsilon\delta} d'(x,\cL)^{-\epsilon\delta},
\end{displaymath}
and thus
\begin{displaymath}
e^{-m' \epsilon \delta} S_1(x) \le S_1(a_t x) \le e^{m' \epsilon \delta} S_1(x)
\end{displaymath}
Then, using (\ref{eq:log-unif}), 
\begin{displaymath}
e^{-m' \epsilon \delta -m/2} S_1(x)u(x)^{1/2} \le S_1(a_t x) u(a_t
x)^{1/2} \le e^{m' \epsilon \delta + m/2} S_1(x) u(x)^{1/2}. 
\end{displaymath}
Suppose $\cL \in \Delta_2 \cup \Delta_3$. 
Then, by (\ref{eq:log-unif}) and (\ref{eq:growth:dprime}), 
\begin{displaymath}
d'(x,\cL) \ge C u(x)^{-k},
\end{displaymath}
where $C = O(1)$ (depending only on $k$, $m$ and $m'$), and thus, for
$i=2,3$, and using Lemma~\ref{lemma:def:J:Jprime},
\begin{displaymath}
S_i(a_t x) \le C n(\cM) u(x)^{-\epsilon \delta k} \quad\text{ and } \quad
S_i(x)
\le C n(\cM) u(a_t x)^{-\epsilon \delta k}, \qquad \text{ $i=2,3$} 
\end{displaymath}
Now choose $\epsilon > 0$ so that $k \epsilon\delta < 1/2$ and
$\lambda > 0$ so that $\lambda > 10 C e^{m} n(\cM)$. Then, 
\begin{displaymath}
S_i(a_t x) u(a_t x)^{1/2} 
\le (0.1) \lambda u(x) \quad\text{ and } \quad S_i(x) u(x)^{1/2}
\le (0.1) \lambda  u(a_t x), \qquad \text{ $i=2,3$}
\end{displaymath}
Then, 
\begin{align*}
f_\cM(a_t x) & = S_1(a_t x)u(a_t x)^{1/2} + S_3(a_t x) u(a_t x)^{1/2}
+ \lambda u(a_t x) &&  \\
& \le  e^{m' \epsilon \delta +m/2} S_1(x) u(x)^{1/2} + (0.1) \lambda
u(x) + e^{m} 
\lambda u(x) && \text{ by (\ref{eq:log-unif}) and (\ref{eq:growth:dprime})} \\
& \le (e^{m' \epsilon \delta +m/2} + (0.1) + e^m) (S_1(x) u(x)^{1/2}
+ \lambda u(x)) && \\
& \le (e^{m' \epsilon + \delta m/2} + (0.1) + e^m) f_\cM(x). 
\end{align*}
In the same way,
\begin{align*}
f_\cM(x) & = S_1(x)u(x)^{1/2} + S_2(x) u(x)^{1/2}
+ \lambda u(x) &&  \\
& \le  e^{m' \epsilon \delta +m/2} S_1(a_t x) u(a_t x)^{1/2} + (0.1) \lambda
u(a_t x) + e^{m} 
\lambda u(a_t x) && \text{ by (\ref{eq:log-unif}) and
  (\ref{eq:growth:dprime})} \\ 
& \le (e^{m' \epsilon \delta +m/2} + (0.1) + e^m) (S_1(a_t x) u(a_t x)^{1/2}
+ \lambda u(a_t x)) && \\
& \le (e^{m' \epsilon + \delta m/2} + (0.1) + e^m) f_\cM(a_t x). 
\end{align*}
\end{proof}

We now begin the verification of condition (b) of
Lemma~\ref{lemma:abstract:recurrence:properties}. The first step is 
the following:
\begin{claim}
\label{claim:fN}
Suppose $\epsilon$ is sufficiently small (depending on $k$, $\delta$). Then
there exist $t_2 > 0$ and $\tilde{b} > 0$ such
that for all $x \in \cH_1(\alpha)$ and all $t > t_2$, 
\begin{multline}
\label{eq:claim:fN}
A_t (s_{\cM,\epsilon}u^{1/2})(x) \le 
\kappa_1(x,t)^{\epsilon}
\tilde{c}^{1/2} s_{\cM,\epsilon}(x) u(x)^{1/2} +
\kappa_1(x,t)^\epsilon \tilde{b}^{1/2} s_{\cM,\epsilon}(x)+ \\ +
b_3(t) n(\cM) u(x), 
\end{multline}
where $\tilde{c} = e^{-\tilde{\eta} t}$ and
$\kappa_1(x,t)$ is as in Lemma~\ref{lemma:ineq-rel}.  
\end{claim}

\bold{Remark.} 
The proof of Claim~\ref{claim:fN} is a straighforward
verification, where we again have to show that contribution of the
subspaces which contribute at $x$ but not at $a_t r_\theta x$ (or vice
versa) is negligible (or more precisely can be absorbed into the
right-hand-side of (\ref{eq:claim:fN})). The main feature of
(\ref{eq:claim:fN}) is the appearance of the ``cross term''
$\kappa_1(x,t)^\epsilon \tilde{b}^{1/2} s_{\cM,\epsilon}(x)$. In order
to proceed further, we will need to show (for a properly chosen $t$), 
that for all $x \in \strat1$, 
$\kappa_1(x,t)^{\epsilon} \tilde{b}^{1/2} \le (0.1) c_0
u(x)^{1/2}$, where $c_0$ is in
Lemma~\ref{lemma:abstract:recurrence:properties} (b). This will be 
done, on a case by case basis, 
in the proof of Proposition~\ref{prop:fM:inequality} below. 

\begin{proof}[Proof of Claim~\ref{claim:fN}] 
In this proof, the $b_i(t)$ denote constants depending on $t$. 
Choose $\epsilon > 0$ so that $2 k \epsilon \delta\le 1$. 
Suppose $t > 0$ is fixed.
Let $J'(x) \subset J_{k,\cM}(x)$ be the subset
\begin{displaymath}
J'(x) = \{ \cL \st a_t r_\theta \cL \in J_{k,\cM}(a_t r_\theta x)
\text{ for all } 0 \le \theta \le 2\pi \}. 
\end{displaymath}
Suppose $\cL \subset J'(x)$. 
For $0 \le \tau \le t$ and $0\le \theta \le 2\pi$, let
\begin{displaymath}
\ell_\cL(a_\tau r_\theta x) = d'(a_\tau r_\theta \cL, a_\tau r_\theta
x)^{-\delta}.  
\end{displaymath}
Then, 
\begin{align}
\label{eq:At:ell:L}
A_t(\ell_\cL^{2\epsilon})(x) 
& \le (A_t \ell_\cL)^{2\epsilon}(x) &&
\text{ by   (\ref{eq:At:Jensen}) } \notag \\
& \le (\kappa_1(x,t) \ell_\cL(x) + b(t))^{2\epsilon} 
&& \text{ by Lemma~\ref{lemma:ineq-rel}} \notag \\
& \le \kappa_1(x,t)^{2\epsilon} \ell_\cL(x)^{2\epsilon}+
b(t)^{2\epsilon} && \text{ by   (\ref{eq:a:plus:b:epsilon})} 
\end{align}
Recall that
\begin{equation}
\label{eq:u:at:least:2}
u(x) \ge 2 \qquad \text{for all $x$}. 
\end{equation}
We have, at the point $x$, 
\begin{align}
\label{eq:long:align}
A_t  (\ell_\cL^{\epsilon} u^{1/2}) & \le (A_t
  \ell_\cL^{2\epsilon})^{1/2}(A_t u)^{1/2} && \text{ by
    Cauchy-Schwartz} \notag \\
 & \le [(\kappa_1(x,t)^{2\epsilon} \ell_\cL(x)^{2\epsilon}  +
b_1(t) u(x)]^{1/2} (\tilde{c} u(x) + \tilde{b})^{1/2}
 && \text{ by (\ref{eq:At:ell:L}), (\ref{eq:sup-harm-u}), 
      (\ref{eq:u:at:least:2}) }\notag \\ 
& \le [\kappa_1(x,t)^{\epsilon} \ell_\cL(x)^{\epsilon} +
b_1(t)^{1/2} u(x)^{1/2}] (\tilde{c}^{1/2} u(x)^{1/2} +
\tilde{b}^{1/2}) && \text{ by (\ref{eq:a:plus:b:epsilon})} 
\notag \\
& = \kappa_1(x,t)^{\epsilon}
\ell_\cL(x)^{\epsilon} (\tilde{c}^{1/2} u(x)^{1/2} + \tilde{b}^{1/2})
&& \notag \\
& \qquad \qquad \qquad \qquad + b_1(t)^{1/2} \tilde{c}^{1/2} u(x) +
b_1(t)^{1/2} 
\tilde{b}^{1/2} u(x)^{1/2} &&  \notag \\
& \le \kappa_1(x,t)^{\epsilon} 
\ell_\cL(x)^{\epsilon} (\tilde{c}^{1/2} u(x)^{1/2} + \tilde{b}^{1/2})+ \notag \\
& \qquad \qquad \qquad \qquad \qquad \qquad +b_1(t)^{1/2} (\tilde{c}^{1/2} + \tilde{b}^{1/2}) u(x) && \text{
  since $u(x) \ge 1$}  \notag \\
& = \kappa_1(x,t)^\epsilon \tilde{c}^{1/2} \ell_\cL(x)^\epsilon
u(x)^{1/2} + 
\kappa_1(x,t)^\epsilon \tilde{b}^{1/2}\ell_\cL(x)^{\epsilon} +  \\ 
& \qquad \qquad \qquad \qquad \qquad \qquad \qquad \qquad \qquad + b_2(t)u(x). &&  \notag 
\end{align}
For $0 \le \tau \le t$ and $0\le \theta \le 2\pi$, let
\begin{displaymath}
h(a_\tau r_\theta x) = \sum_{\cL \in J'(x)} 
d'(a_\tau r_\theta \cL, a_\tau r_\theta x)^{-\epsilon\delta} =
\sum_{\cL \in J'(x)} \ell_\cL(a_\tau r_\theta x)^{\epsilon}. 
\end{displaymath}
Then, $h(a_\tau r_\theta x) \le
s_{\cM,\epsilon}(a_\tau r_\theta x)$. Summing (\ref{eq:long:align})
over $\cL \in J'(x)$ and using Lemma~\ref{lemma:def:J:Jprime} we get
\begin{equation}
\label{eq:estimate:At:hM}
A_t (h u^{1/2})(x) \le \kappa_1(x,t)^{\epsilon}
\tilde{c}^{1/2} h(x) u(x)^{1/2} +
\kappa_1(x,t)^\epsilon \tilde{b}^{1/2} h(x) +
b_2(t)n(\cM) u(x)   
\end{equation}
We now need to estimate the contribution of subspaces not in $J'(x)$. 
Suppose $0 \le \theta \le 2\pi$, and suppose
\begin{displaymath}
a_t r_\theta \cL \in J_{k, \cM}(a_t r_\theta x),
\quad \text{  but   $\cL\not\in J'(x)$. }
\end{displaymath}
Then, either $\cL \not\in J_{k,\cM}(x)$ or for some $0 \le \theta' \le
2\pi$, $a_t r_{\theta'} \cL \not\in J_{k,\cM}(a_t r_{\theta'}
x)$. Then in either case, for some $\tau' \in \{0,t\}$ and 
some  $0 \le \theta'\le 2\pi$, 
$a_{\tau'} r_{\theta'} \cL \not\in
J_{k,\cM}(a_{\tau'} r_{\theta'} x)$. Hence
\begin{displaymath}
d'(a_{\tau'} r_{\theta'} x, a_{\tau'} r_{\theta'} \cL) \ge u(a_{\tau'}
r_{\theta'} x)^{-k}
\end{displaymath}
Then,  by (\ref{eq:growth:dprime}) and (\ref{eq:log-unif}),
\begin{displaymath}
d'(x, \cL) \ge b_0(\tau')^{-1} u(x)^{-k} \ge b_0(t)^{-1} u(x)^{-k}
\end{displaymath}
and thus, for all $\theta \in [0,2\pi]$, by (\ref{eq:log-unif}) and
(\ref{eq:growth:dprime}), 
\begin{displaymath}
d'(a_t r_\theta x, a_t r_\theta \cL) \ge b_0(t)^{-2} u(x)^{-k}.
\end{displaymath}
Hence, using (\ref{eq:log-unif}) again, 
\begin{equation}
\label{eq:estimate:term:outside:hm}
d'(a_t r_\theta x, a_t r_\theta \cL)^{-\epsilon\delta} u(a_t r_\theta
x)^{1/2}  \le
b_1(t) u(x)^{k\epsilon\delta+1/2} \le b_1(t) u(x),  
\end{equation}
where for the last estimate we used $k \epsilon\delta \le 1/2$.
Thus, for all $0 \le \theta \le 2\pi$, 
\begin{align*}
s_{\cM,\epsilon}(a_t r_\theta x) u(a_t r_\theta x)^{1/2} & \le
h(a_t r_\theta x) u(a_t r_\theta x)^{1/2} + |J(a_\tau 
r_\theta x)| b_1(t) 
u(x) && \text{ using (\ref{eq:estimate:term:outside:hm})} \\
& \le h(a_t r_\theta x) u(a_t r_\theta x)^{1/2} + b_1(t) n(\cM) u(x)
 && \text{ using Lemma~\ref{lemma:def:J:Jprime}.} 
\end{align*}
Hence, 
\begin{align}
A_t & (s_{\cM,\epsilon} u^{1/2})(x) \le A_t (h u^{1/2})(x) +
b_1(t) n(\cM) u(x) 
 && \notag \\
& \le \kappa_1(x,t)^{\epsilon}
\tilde{c}^{1/2} h(x) u(x)^{1/2} +
\kappa_1(x,t)^\epsilon \tilde{b}^{1/2} h(x)+ b_3(t) n(\cM) u(x) 
 && \text{using (\ref{eq:estimate:At:hM}) } \notag \\
& \le \kappa_1(x,t)^{\epsilon}
\tilde{c}^{1/2} s_{\cM,\epsilon}(x) u(x)^{1/2} +
\kappa_1(x,t)^\epsilon \tilde{b}^{1/2} s_{\cM,\epsilon}(x)+ 
b_3(t) n(\cM) u(x) 
 && \text{since $h \le s_{\cM,\epsilon}$ } \notag
\end{align}
\end{proof}

\begin{proof}[Proof of Proposition~\ref{prop:fM:inequality}]
Let $\sigma$ be as in Claim~\ref{claim:fM:condition:a:works}, and 
let $c_0 = c_0(\sigma)$ be as in
Lemma~\ref{lemma:abstract:recurrence:properties}. 
Let $L_0$, $\eta_0'$, $\eta_3$, $m'$, $\delta$ 
be as in Lemma~\ref{lemma:ineq-rel}.
Suppose $\epsilon > 0$ is small enough so that 
\begin{equation}
\label{eq:epsilon:cond:one}
\epsilon m' \delta <
\frac{1}{2} \tilde{\eta}, 
\end{equation}
where 
$\tilde{\eta}$ is as in Theorem~\ref{theorem:alpha-fun}. We also
assume that $\epsilon > 0$ is small enough so that 
\begin{equation}
\label{eq:epsilon:cond:two}
\epsilon m' \delta < \frac{1}{2} \min(\eta_3,\eta_0')
\end{equation}
where $\eta_3$ is in Lemma~\ref{lemma:ineq-rel}. Choose $t_0 > 0$ so that
Theorem~\ref{theorem:alpha-fun} holds  for $t > t_0$, 
and so that $e^{-\tilde{\eta} t_0}
< (0.1) c_0$. Since $\kappa_1(x,t) < e^{m' \delta t}$, 
we can also, in view of (\ref{eq:epsilon:cond:one})
make sure that for $t > t_0$, 
\begin{equation}
\label{eq:choice:tilde:c}
\kappa_1(x,t)^\epsilon e^{-\tilde{\eta}t/2} \le (0.1)c_0
\end{equation}
Let $t_2 > 0$ be such that Claim~\ref{claim:fN} holds. By
(\ref{eq:epsilon:cond:two}), there exists $t_3 > 0$ so that for $t >
t_3$, 
\begin{equation}
\label{eq:tmp:cond:kappa1}
\kappa_1(x,t)^\epsilon \tilde{b}^{1/2} \le e^{m' \delta \epsilon
  t} \tilde{b}^{1/2} \le (0.1)  c_0 e^{\eta_0' t/2} 
\end{equation}
By Lemma~\ref{lemma:ineq-rel} there exists $\tau > \max(t_0, t_2,t_3)$
such that for all $x$ with $\log u(x) < L_0 + \eta_0' \tau$, 
\begin{displaymath}
\kappa_1(x,\tau)^{\epsilon} \tilde{b}^{1/2} \le
(0.1) c_0 \le (0.1) c_0 u(x)^{1/2}.  
\end{displaymath}
If $\log u(x)  \ge L_0 + \eta_0' \tau$, then $u(x)^{1/2} \ge e^{(\eta_0'/2)
\tau}$, and therefore, since $\tau > t_3$, by (\ref{eq:tmp:cond:kappa1}), 
\begin{displaymath}
\kappa_1(x,\tau)^\epsilon \tilde{b}^{1/2} \le e^{m' \delta \epsilon
  \tau} \tilde{b}^{1/2} \le (0.1)  c_0 e^{\eta_0' \tau/2} \le (0.1)
c_0 u(x)^{1/2}. 
\end{displaymath}
Thus, for all $x \in \strat1$, 
\begin{equation}
\label{eq:estimate:kappa1:tilde:b}
\kappa_1(x,\tau)^{\epsilon} \tilde{b}^{1/2} \le (0.1) c_0 u(x)^{1/2}.  
\end{equation}
Thus, substituting (\ref{eq:choice:tilde:c}) and
(\ref{eq:estimate:kappa1:tilde:b}) into (\ref{eq:claim:fN}),
we get, for all $x \in \cH_1(\alpha)$, 
\begin{equation}
\label{eq:tmp:At:lcN:uhalf}
A_\tau (s_{\cM,\epsilon} u^{1/2})(x) \le 
(0.2)c_0  \, s_{\cM,\epsilon}(x) u(x)^{1/2} +
b_3(\tau) n(\cM) u(x).   
\end{equation}
Choose 
\begin{displaymath}
\lambda > 10  b_3(\tau) n(\cM)/c_0. 
\end{displaymath}
Then, in view of (\ref{eq:tmp:At:lcN:uhalf}), we have
\begin{equation}
\label{eq:At:lcN:uhalf}
A_\tau (s_{\cM,\epsilon} u^{1/2})(x) \le (0.2) c_0 \, s_{\cM,\epsilon}(x)
u^{1/2} + (0.1) c_0 \lambda u(x).
\end{equation}
Finally, since $\tilde{c} \le (0.1)c_0$, we have
\begin{align*}
A_\tau (f_\cM)(x) & = A_\tau (s_{\cM,\epsilon} u^{1/2})(x) + A_\tau (\lambda
u)(x) && \\
& \le [(0.2) c_0 s_{\cM,\epsilon}(x)
u^{1/2} + (0.1) c_0 \lambda u(x)] + (0.1) c_0 \lambda u(x) + \lambda \tilde{b}
&& \text{ by (\ref{eq:At:lcN:uhalf}) and
  (\ref{eq:sup-harm-u})}
 \\
& \le (0.2) c_0 f_\cM(x) + b_\cM && \text{ where $b_\cM = \lambda \tilde{b}$.}
\end{align*}
Thus, condition (b) of
Lemma~\ref{lemma:abstract:recurrence:properties} holds for $f_\cM$. In
view of Lemma~\ref{lemma:abstract:recurrence:properties} this
completes the proof of Proposition~\ref{prop:fM:inequality}. 
\end{proof}


\section{Countability}
\label{sec:countability}

The following lemma is standard: 
\begin{lemma}
\label{lemma:in:L1}
Suppose $SL(2,\reals)$ acts on a space $X$, and suppose
there exists a proper function 
$f: X \to [1,\infty]$ such that 
{}{
for some $\sigma > 1$ all $0 \le t \le 1$ and all $x \in X$, 
\begin{displaymath}
\sigma^{-1} f(x) \le f(a_t x) \le \sigma f(x),
\end{displaymath}
and also there exist $0 < c < c_0(\sigma)$ (where $c_0(\sigma)$ is as
in Lemma~\ref{lemma:abstract:recurrence:properties}), $t_0 > 0$ 
and $b > 0$   such that for all $t > t_0$ and all $x \in X$, 
\begin{displaymath}
A_t f(x) \le c f(x) + b,
\end{displaymath}
}
Suppose $\nu$ is an ergodic $SL(2,\reals)$-invariant measure on $X$, such that
$\nu(\{ f < \infty \}) > 0$. 
Then, 
\begin{equation}
\label{eq:upper:bound:integral}
\int_X f \, d\nu \le B,
\end{equation}
where $B$ depends only on $b$, $c$ and $\sigma$.  
\end{lemma}

\begin{proof} For $n \in \natls$, let $f_n = \min(f,n)$.
By the Moore ergodicity theorem, the action of $A \equiv \{ a_t \st t
\in \reals \}$ on $X$ is ergodic. Then, by the Birkhoff ergodic
theorem, there exists a point $x_0 \in
X$ such that for almost all $\theta \in [0,2\pi]$ and all $n \in
\natls$, 
\begin{equation}
\label{eq:tmp:birkhoff:1}
\lim_{T\to \infty} \frac{1}{T}\int_0^T f_n(a_t r_\theta x_0) \, dt =
\int_X f_n \, d\nu
\end{equation}
Therefore for each $n$ there exists a subset $E_n \subset [0,2\pi]$
{}{of measure at least $\pi$}
such that the convergence in (\ref{eq:tmp:birkhoff:1}) is uniform over
$\theta \in E_n$. Then there exists $T_n > 0$ such that for all $T > T_n$,
\begin{equation}
\label{eq:tmp:birkhoff:2}
\frac{1}{T}\int_0^{T} f_n(a_t r_\theta x_0) \, dt \ge \frac{1}{2}
\int_X f_n \, d\nu \quad\text{for $\theta \in E_n$.} 
\end{equation}
We integrate (\ref{eq:tmp:birkhoff:2}) over $\theta \in [0,2\pi]$. Then for
all $T > T_n$, 
\begin{equation}
\label{eq:tmp:birkhoff:3}
\frac{1}{T}\int_0^{T} \left(\int_{0}^{2\pi} f_n(a_t r_\theta x_0) \,
d\theta \right)\, dt \ge \frac{1}{4} \int_X f_n \, d\nu
\end{equation}
But, by Lemma~\ref{lemma:abstract:recurrence:properties} (ii), for
sufficiently large $T$, the
integral in parenthesis on the left hand side of (\ref{eq:tmp:birkhoff:3})
is bounded above by $B' = B'(c,b,\sigma)$. 
Therefore, for all $n$,
\begin{displaymath}
\int_X f_n \, d\nu \le 4B'
\end{displaymath}
Taking the limit as $n \to \infty$ we get that $f \in L^1(X,\nu)$ and 
(\ref{eq:upper:bound:integral}) holds.
\end{proof}

\begin{proof}[Proof of Proposition~\ref{prop:countability}]
Let $X_d(\alpha)$ denote the set of affine manifolds of dimension
$d$. It enough to show that each $X_d(\alpha)$ is countable. 

For an affine subspace $\cL \subset \rhr$ whose linear part is $L$, 
let $H_\cL: p(L) \to \ker p/(L
\cap \ker p)$ denote the linear map such that for $v \in p(L)$, $v +
H_\cL(v) \in L \mod L \cap \ker p$. For an affine manifold $\cM$, let 
\begin{displaymath}
H(\cM) = \sup_{x \in \cM \cap K'} \|H_{\cM_x}\|_x
\end{displaymath}
where we use the notation $\cM_x$ for the affine subspace tangent to
$\cM$ at $x$. 

For an integer $R > 0$, let
\begin{displaymath}
X_{d,R}(\alpha) = \{ \cM \in X_{d}(\alpha) \st n(\cM) \le R
\text{ and } H(\cM) \le R \}. 
\end{displaymath}
Since $X_{d}(\alpha) = \bigcup_{R=1}^\infty X_{d,R}(\alpha)$, it is
enough to show that each $X_{d,R}(\alpha)$ is finite. 

Let $K'$ be as in Definition~\ref{def:number:sheets} of
$n(\cdot)$,  and let $L_R(K')$ denote the set of (unordered) $\le
R$-tuples of $d$ 
dimensional affine subspaces intersecting $K'$. Then $L_R(K')$ is
compact, and we have the map $\phi: X_{d,R} \to L_R(K')$ which takes
the affine manifold $\cM$ to the (minimal) set of affine subspaces containing
$\cM \cap K'$. 

Suppose $\cM_j  \in X_{d,R}(\alpha)$ is an infinite sequence, with
$\cM_j \ne \cM_k$ for $j \ne k$. Then, $\cM_j \cap K' \ne \cM_k \cap
K'$ for $j \ne k$. (If $\cM_j \cap K' = \cM_k \cap K'$ then by the
ergodicity of the $SL(2,\reals)$ action, $\cM_j = \cM_k$). 

Since $L_R(K')$ is compact, after
passing to a subsequence, we may assume that $\phi(\cM_j)$
converges. Therefore, 
\begin{equation}
\label{eq:hausdorff:to:0}
hd(\cM_j \cap K', \cM_{j+1} \cap K') \to 0 
\quad\text{as $j \to \infty$},
\end{equation}
where $hd( \cdot, \cdot)$ denotes the Hausdorff distance.
{}{(We use any metric on $\strat1$ for which the period
  coordinates are continuous).}
Then, because of (\ref{eq:hausdorff:to:0}) and the bound on $H(\cM)$,
for all $x \in \cM_{j+1} \cap K'$, $d'(x,\cM_{j}) \to 0$. 
\mc{explain} 
{}{From the definition of $f_\cM$, we have $f_\cM(x) \to
  \infty$ as $d'(x,\cM) \to 0$.} 
Therefore, there exists a sequence $T_j \to \infty$ such
that we have 
\begin{equation}
\label{eq:f:adjacent:bounded}
f_{\cM_{j+1}}(x) \ge T_j \quad\text{for all $x \in \cM_j \cap K'$.}
\end{equation}
Let $\nu_j$ be the affine $SL(2,\reals)$-invariant probability
measure whose support is $\cM_j$. 
Then, by Proposition~\ref{prop:fM:inequality} and
Lemma~\ref{lemma:in:L1}, we have for all $j$, 
\begin{displaymath}
\int_{\strat1} f_{\cM_{j+1}} \, d\nu_j \le B,
\end{displaymath}
where $B$ is independent of $j$. But, {}{by the
  definition of $K'$ and 
Theorem~\ref{theorem:all:measures:return}, 
\begin{displaymath}
\nu_j(\cM_j \cap K') \ge 1-\rho \ge 1/2. 
\end{displaymath}
}
This is a contradiction to
(\ref{eq:f:adjacent:bounded}). Therefore, $X_{d,R}(\alpha)$ is
finite. 
\end{proof}

\end{document}